\documentclass[12pt]{amsart}
\usepackage{amsfonts}
\usepackage{fullpage}
\usepackage{amsmath,amsthm}
\usepackage{latexsym}
\usepackage{array}
\usepackage{dsfont}
\usepackage{amssymb}
\usepackage{graphics}
\usepackage{enumerate}
\usepackage[english]{babel}
\usepackage[applemac]{inputenc}

\usepackage[T1]{fontenc}
\usepackage{color}
\definecolor{marin}{rgb}   {0.,   0.3,   0.7} 
\definecolor{rouge}{rgb}   {0.8,   0.,   0.} 
\definecolor{sepia}{rgb}   {0.8,   0.5,   0.} 
\usepackage[colorlinks,citecolor=marin,linkcolor=rouge,
            bookmarksopen,
            bookmarksnumbered
           ]{hyperref}

\newtheorem{lemma}{Lemma}[section]

\newtheorem{theorem}[lemma]{Theorem}
\newtheorem{proposition}[lemma]{Proposition}
\newtheorem{corollary}[lemma]{Corollary}
\newtheorem{remark}[lemma]{Remark}
\newtheorem{example}[lemma]{Example}

\newtheorem{notation}[lemma]{Notation}
\newtheorem{definition}[lemma]{Definition}
\newtheorem{conclusion}[lemma]{Conclusion}
\numberwithin{equation}{section}
\newcommand{\QED}{\mbox{}\hfill \raisebox{-0.2pt}{\rule{5.6pt}{6pt}\rule{0pt}{0pt}} 
          \medskip\par}

\newcommand{\rad}{{\mathrm{rad}}}

\newcommand{\dd}{\mathrm{d}}

\newcommand{\Ec}{\mathcal{E}}
\newcommand{\Hc}{\mathcal{H}}

\newcommand{\N}{\mathbb{N}}

\newcommand{\R}{\mathbb{R}}
\newcommand{\Fc}{\mathcal{F}}

\newcommand{\C}{\mathbb{C}}

\newcommand{\T}{\mathbb{T}}

\newcommand{\Sc}{\mathcal{S}}

\newcommand{\Z}{\mathbb{Z}}
\newcommand{\Tc}{\mathcal{T}}

\newcommand{\Norm}[2]{\|#1\|\left.\vphantom{T_{j_0}^0}\!\!\right._{#2}}

\author{Erwan Faou}
\address{Univ Rennes, INRIA, CNRS, IRMAR - UMR 6625, F-35000 Rennes, France } 
\email{Erwan.Faou@inria.fr}

\author{ Pierre Rapha\"el }
\address{ Department of Pure Mathematics and Mathematical Statistics, Centre for Mathematical Sciences, Cambridge, UK. }
  \email{pr463@dpmms.cam.ac.uk}

\title[]
{ On weakly turbulent solutions to the perturbed linear Harmonic oscillator }

\begin{document}

\begin{abstract} We introduce specific solutions to the linear harmonic oscillator, named bubbles. They form resonant families of invariant tori of the linear dynamics, with arbitrarily large Sobolev norms. We use these modulated bubbles of energy to construct a class of potentials which are real, smooth, time dependent and uniformly decaying to zero with respect to time, such that the corresponding perturbed quantum harmonic oscillator admits solutions which exhibit a logarithmic growth of Sobolev norms. The resonance mechanism is explicit in space variables and produces highly oscillatory solutions. We then give several recipes to construct similar examples using more specific tools based on the continuous resonant (CR) equation in dimension two. \end{abstract}

\subjclass{  }
\thanks{
}

\maketitle



\section{Introduction}



\subsection{Setting of the problem}


We consider the linear operator associated with the two dimensional quantum harmonic oscillator
\begin{equation}
\label{defH}
H = - \Delta + |x|^2, 
\end{equation}
where for $x = (x_1,x_2) \in \R^2$, $|x|^2 = x_1^2 + x_2^2$ and $\Delta$ the Laplace operator. 
Let the Sobolev norms associated with the function space defining the domain of $H$
$$
\Hc^r = \{ u \in L^2 \, | \,H^{r/2} f \in L^2\}, \quad r \geq 0, 
$$
then the solution to the linear Schr\"odinger equation 
\begin{equation}
\label{J}
\left|\begin{array}{l}
i \partial_t u = Hu, \qquad u(t,x) \in \C\\
u_{|t=0}=u_0(x)
\end{array}\right. \quad \Leftrightarrow \quad u(t,x) = e^{-i t H} u_0
\end{equation}
preserves all the $\Hc^r$ norms $$\forall t\in \Bbb R, \ \ \Norm{e^{- i t H} u_0}{\Hc^r} = \Norm{u_0}{\Hc^r}$$ and no {\em weakly turbulent} effect can be observed, {\em i.e.}\ energy transfer between low and high frequencies generating growth of Sobolev norms.\\

A long standing open problem is the possibility of finding perturbations of \eqref{J} of the Hamiltonian form
\begin{equation}
\label{cenioneivbneivoeb}
i \partial_t u = Hu + V(t,x,u) u
\end{equation}
producing such weakly turbulent effects, while preserving energies ($L^2$ norm and/or Hamiltonian energy in the time independent case), and to classify possible mechanisms of energy transfers, as well as their genericity. We propose in this paper a step forward in this direction by considering the linear case where the  real potential $V(t,x)$  is independent of $u$ and chosen as smooth and small as possible. We in purpose focus onto the simplest possible linear case, but insist that the method of proof and the nature of the resonance mechanism will also apply to non linear problems $V(t,x,u) = W(t,x)+f(|u|^2)$. 


\subsection{Previous results}


The study of the linear Schr\"odinger equation \eqref{J} perturbed by a general time dependent linear operator $ P(t ) u$ (not necessarily the multiplication with a function) has a long history with important recent developments.\\

The first class of results exhibit situations where solutions do not have any turbulent behavior and remain bounded for all times. The perturbed flow is essentially similar to the unperturbed one and the dynamics can be conjugated to a dynamics with a constant linear operator close to $H$. These are reducibility results generalizing Floquet theory and the perturbation operator is typically periodic or quasi-periodic in time. In \cite{Com87} such results were given for regularizing perturbation. Using KAM technics for PDEs more recent results have been shown, see for instance \cite{BG01, GT11, Bam17a, Bam17b, BGMR18,GP19}.\\ 

A second class of results concerns a priori bounds on the possible growth of Sobolev norms. In \cite{MR17}, general bounds in times where given for the case where the perturbation is a multiplication by a real potential $V(t,x)$. When the potential is regular, the growth can be at most of order $t^{\varepsilon}$ where $\varepsilon$ depends on the regularity of the potential, a bound that can be refined to $(\log t)^{\alpha}$ for analytic potential. More general results are also given in \cite{BGMR17}. \\

Concerning the possibility of growth and the existence of weakly turbulent mechanisms, very few results are available. In \cite{GY00}, explicit examples are given with solutions 
exhibiting Sobolev norm growth, and an explicit multiplication operator $V(t,x) = a \sin(t) x$. Note however that this operator is of order $1$, and in particular not decaying at infinity in $x$ and thus not defining an element of $\Hc^r$. 
In \cite{Del14}, J.-M. Delort constructed order zero pseudo differential operators $P(t)$ periodic in time, and such that the solution of the Harmonic oscillator perturbed by this operator growth like $t^{r/2}$ in $\Hc^r$ norm. 
Similar examples were given by A. Maspero \cite{Mas18}. During the preparation of this work, L. Thomann \cite{Tho20} also proposed an example of such operators based on a linearized version of the lowest Landau level equation, and constructed as explicit travelling waves. All these examples provide continuous operators $P(t)$ of order $0$ with periodic or growing behavior with respect to $t$, when $P(t)$ and its time derivatives are estimated in $\mathcal{L}(\Hc^r,\Hc^r)$, but so far no result has been given with the multiplication by a smooth potential belonging to $\Hc^r$.\\

Spectacular results have also been obtained in the case of the torus regarding a priori bounds, reducibility, and the construction of unbounded trajectories \cite{Bou99a,Bou99b}, see also \cite{Wan08b, Del10, MR17, EK09}. Specifically in the non linear setting of the (NLS) equation on the torus, the seminal work \cite{CKSTT10} provides the first explicit construction of growth mechanism for the limiting completely resonant equation. This analysis was refined in \cite{GK13} with optimized constants, and  used in \cite{HPTV15} to show the relevance of the mechanism for the small data scattering problem.\\

More growth mechanisms have also been explored for other non linear dispersive problems in particular in \cite{GG10, Po11, GLPR18} which are deeply connected to our approach.


\subsection{Statement of the result}


We propose a new and elementary space based approach to construct classes of smooth asympotically in time vanishing potentials $V(t,x)$ for which a weakly turbulent
mechanism for solutions to \eqref{cenioneivbneivoeb} occurs. Our construction comes with a complete description of the associated drift to high frequencies. Our main result is the following.

\begin{theorem}[Existence of smooth vanishing potentials exhibiting weakly turbulent growth]
\label{Th1}
There exist potential functions $V(t,x) \in \mathcal{C}^\infty(\R\times \R^2; \R)$ and functions $u(t,x) \in \mathcal{C}^\infty(\R\times \R^2; \C)$ such that 
\begin{equation}
\label{LHO1}
\forall\, (t,x) \in \R \times \R^2\qquad i \partial_t u(t,x) = H u (t,x) + V(t,x) u(t,x) 
\end{equation}
and such that 
for all $r \geq 0$ and all $k \in \N$
\begin{equation}
\label{decay}
\lim_{t \to +\infty} \Norm{\partial_t^{k} V(t,x)}{\Hc^r} = 0
\end{equation}
and
\begin{equation}
\label{growth}
\Norm{u(t,x)}{\Hc^1} \sim c (\log t)^\alpha, \quad t \to +\infty, \quad c,\alpha >0.  
\end{equation}
\end{theorem}

\noindent{\em Comments on the result.}\\

\noindent{\em 1. The bubble approach}. The main ingredient used to prove this result is the study of specific solutions to the unperturbed equation \eqref{J} that we call {\em bubbles}. They are explicit solutions whose trajectories form families of invariant tori of dimension one, parametrized by {\em actions} piloting the $\Hc^1$ norm of the solution, and with {\em angles all oscillating  at the same frequency} corresponding to the frequency gap of the operator $H$. They thus form a resonant family of invariant tori of the linear dynamics, with arbitrarily high Sobolev norms. We then construct the perturbation as superposition of time oscillations which resonate with the bubbles decaying for large time to produce a growth of the $\Hc^1$ norm corresponding to a growth of the actions in the family of bubbles. A fundamental feature is that the bubbles are completely explicit and generated by the {\em pseudo conformal symmetry group} associated to the unperturbed flow \eqref{J}, and the leading order growth mechanism corresponds to a suitable resonant mechanism created by a fine tunning of the potential $V(t,x)$. In other words, {\em after renormalization}, the growth of Sobolev norms is generated by a small deformation of a solitary wave (here just a harmonic function). This is the heart of the analysis of blow up bubbles for (NLS) models in \cite{MeRa05, MaRa18} and the study of growth mechanisms in \cite{GLPR18}. Let us stress that the mechanism is completely explicit and \eqref{vbeivbibeveievb} gives an example of such an admissible potential in closed form.\\

\noindent{\em 2. Modulation equations and Arnold diffusion.} Resonance will be described through the study of modulation equations which are a perturbation of the trajectory associated to the pseudo-conformal symmetry of \eqref{defH}, section \ref{sectionmodualtib}. The obtained growth mechanism is deeply connected to the original example of {\em Arnold diffusion} given in \cite{Arn64} (see \cite{DGLS08} for a review on the subject). Indeed the modulation equations describing the evolution of the bubble in interaction with the complete system is a perturbation of a completely integrable system (see \eqref{arnold} below) containing resonant oscillations as in \cite{Arn64}, but of size $\varepsilon$ decaying in time in a non integrable way. Compared with the classical result in Arnold diffusion, this class of perturbations allows a complete growth in infinite time of the actions  at a logarithmic scale. Moreover, as these bubbles are embedded into an infinite dynamical system, we construct the solution by superposing this new Arnold diffusion example with the backward integration methods for PDEs introduced in \cite{Me90}. 
\\

\noindent{\em 3 Oscillations}. An essential difference with the blow up analysis in  \cite{MeRa05, MaRa18,GLPR18} is the {\em oscillatory nature} of the corresponding solutions which are a consequence of the discrete spectrum of the operator. For example, for the solution contructed in Theorem \ref{Th1}, there exist $t_n^{(1)},t_n^{(2)}\to +\infty$ such that :$$\left|\begin{array}{l}
\lim_{n\to +\infty}\|u(t^{(1)}_n,\cdot)\|_{L^\infty}=0\\
\lim_{n\to +\infty}\|u(t^{(2)}_n,\cdot)\|_{L^\infty}=+\infty.
\end{array}\right.
$$ Monotonic growth of the energy is however achieved at the level of the action-angles variables which is the core of the resonant mechanism. We refer to \cite{MRRS19} for more highly oscillatory blow up mechanisms for (NLS) like models.\\

\noindent{\em 4. The growth rate}. Interestingly enough, the logarithmic growth rate \eqref{growth} saturates the general bound for smooth potentials proved in \cite{MR17}. Note that typically in all the examples we construct, we will be able to estimate the growth of higher Sobolev norm of $u$, that will be of order $(\log t)^{r\alpha}$ for the norm $\Hc^r$. Moreover, by tuning differently the potential, we can also produce bounds of order $t^\varepsilon$ but the estimate \eqref{decay} will be valid only up to some $k$ depending on $\varepsilon^{-1}$. These type of refinements and discussions about optimality of the result, as well as a complete classification of the examples yielding to Theorem \ref{Th1} will be out of the scope of this paper.\\

\noindent{\em 5. More growth mechanisms}. In section \ref{sectioncr} we also give general recipes to construct examples realizing Theorem \ref{Th1} for the pseudo-differential linearized CR equation introduced in \cite{FGH16} which is the first normal form operator of the cubic nonlinear Harmonic oscillator as shown in \cite{GHT16}.  The strategy here is in some sense closer to \cite{Tho20} who considers the specific case of the Bargmann-Fock space, but turns out to be in fact very general. \\

\medskip 

\noindent {\bf Acknowledgments}.  This work was completed during the participation of E.F. to the semester {\em Geometry, compatibility and structure preservation in computational differential equations} held at the Isaac Newton Institute, Cambridge, in Fall 2019. This visit was partially supported by a grant from the Simons Foundation.  P.R. is supported by the ERC-2014-CoG 646650 SingWave. P.R. would like to thank P. Gerard, Z. Hani and Y. Martel for stimulating discussions at very early stages of this work at the 2015 MSRI program "New challenges in PDE". The authors would also like to thank L. Thomann for his careful reading of a preliminary version of the manuscript and his fruitful comments. 


\section{Bestiary}


We recall in this section basic facts about the harmonic oscillator and Hermite functions which will be used in the proof of the main Theorem. We work in all the paper in dimension 2.


\subsection{Notations}


We set 
$$
( f, g )_{L^2} = \int_{\R^2} f(x) \overline{g(x)} \dd x.  
$$
We define the Fourier transform 
$$
\widehat f (\xi) = (\mathcal{F} f)(\xi) := \frac{1}{2\pi} \int_{\R^2} f(x) e^{- i x \cdot \xi } \, \dd x, 
$$ 
with $x \cdot y = x_1 y_1 + x_2 y_2$ for $x = (x_1,x_2) \in \R^2$ and $y = (y_1,y_2) \in \R^2$. 
With this normalization we have 
$$
(\mathcal{F}^{-1} f) (x) = \frac{1}{2\pi} \int_{\R^2} f(\xi) e^{i x \cdot \xi } \, \dd \xi.  
$$
We define $L^2(\R^2)$ the space the Hilbert space based on the scalar product $(\, \cdot \, , \, \cdot \, )_{L^2}$, and the Sobolev space $H^r$ equipped with the norm. For all $r \geq 0$, defining $\langle \nabla \rangle^r f$ as the inverse Fourier transform of the function $\langle \xi \rangle^r \widehat f(\xi)$, where for any complex number $z$, $\langle z \rangle = \sqrt{1 + |z|^2}$.  We set 
\begin{equation}
\label{Sob}
\Norm{f}{H^r} := \Norm{\langle \nabla \rangle^r f}{L^2} =  \Norm{\langle \xi \rangle^r \widehat f}{L^2}. 
\end{equation}
Finally, we will use the following notation
\begin{equation}
\label{eqLambda}
\Lambda = x \cdot \nabla_x = x_1 \partial_{x_1} + x_2 \partial_{x_2}. 
\end{equation}

\subsection{Harmonic oscillator and eigenfunctions}
Following \cite[Section 6.6]{GHT16} inspired by \cite[Chapter 1 \and Corollary 3.4.1]{Tha93} we consider the radial functions 
\begin{equation}
\label{basishn}
h_{k} = \frac{1}{\sqrt{\pi}} L_k^{(0)}(|x|^2) e^{- \frac{|x|^2}{2}}, \qquad L_k^{(0)} (x) = \frac{e^x}{k! } \frac{\dd ^k}{\dd x^k} (e^{-x} x^k). 
\end{equation}
The $L_k^{(0)}$ are the standard Laguerre polynomials on $[0,+\infty]$. Then we have 
\begin{equation}
\label{lambdan}
H h_k = \lambda_k h_k = ( 4 k +2) h_k \quad \mbox{and}\quad \int_{\R^2} h_k(x) h_n(x) \dd x = \delta_{nk},  
\end{equation}
for all $n,k \in \N$, 
where $\delta_{nk} = 0$ for $n \neq k$ and $\delta_{nn} = 1$. 
The familly $\{h_k\}_{k \geq 0}$ forms an $L^2$ orthonormal basis of radial functions in $L^2(\R^2)$. Note that we have $L_0^{(0)}(x) = 1$ and $L_1^{(0)}(x) = 1 - x$. The general expression of the $h_k$ can be computed using generating functions. 
For any complex number $|t| < 1$ the generating function of the Laguerre polynomials is given by 
$$
\sum_{n = 0}^{\infty} t^n L_n(z) = \frac{1}{1 - t} e^{- \frac{t z}{1 - t}}, 
$$
which is valid for $z \in \R$. 
Hence
\begin{equation}
\label{generhn}
\sum_{n = 0}^{\infty} t^n h_n(x) =\frac{1}{\sqrt{\pi}}\frac{1}{1 - t} e^{- \frac{t |x|^2}{1 - t}}e^{- \frac{|x|^2}{2}} =  \frac{1}{\sqrt{\pi}}\frac{1}{1 - t} e^{- \frac{(1+ t)}{2(1 - t)} |x|^2}. 
\end{equation}
We recall the formula for generalized Laguerre polynomials, for $\alpha \in \R$, 
\begin{equation}
\label{temple}
(k+1) L_{k+1}^{(\alpha)} (x) = (2 k +1 + \alpha - x) L_k^{(\alpha)}(x) - (k + \alpha) L_{k-1}^{\alpha}(x), \quad k \geq 1, 
\end{equation}
a formula which is also true for $k = 0$ (with for instance the definition $L_{-1}^{(\alpha)} = 0$. 
We also need the formulas 
$$
\frac{\dd^k}{\dd x^k} L_n^{(\alpha)}(x) = \left\{
\begin{array}{ll}
(-1)^k L_{n-k}^{(\alpha + k)}(x) & \mbox{if}\quad k \leq n,\\
0 & \mbox{otherwise}
\end{array}
\right.
$$
and
$$
x L_{n}^{(\alpha +1)}(x) = (n + \alpha) L_{n-1}^{(\alpha)}(x) - (n-x) L_{n}^{\alpha}(x). 
$$
From these relations, we obtain 
\begin{equation}
\label{temple2}
x \frac{\dd}{\dd x}L_n^{(\alpha)}(x) =  - x L_{n-1}^{(\alpha +1)} = n L_{n}^{(\alpha)} - (n + \alpha) L_{n-1}^{(\alpha)}. 
\end{equation}
From Equation \eqref{temple} with $\alpha = 0$, we infer for $k \geq 0$
\begin{equation}
\label{eqy2}
|x|^2 h_k(x) = -(k+1) h_{k+1}(x)  + (2k +1) h_k(x) - k h_{k-1}(x). 
\end{equation}
This implies that 
\begin{equation}
\label{eqLap}
\Delta h_k(x) = -(k+1) h_{k+1}(x)  - (2k +1) h_k(x) - k h_{k-1}(x). 
\end{equation}
Moreover, with $\Lambda$ given by \eqref{eqLambda} and using \eqref{temple2}
we have 
\begin{eqnarray}
\Lambda h_{k} &=& \frac{2}{\sqrt{\pi}}|x|^2 \frac{\dd}{\dd r}L_k^{(0)}(|x|^2) e^{- \frac{|x|^2}{2}}
- \frac{1}{\sqrt{\pi}} |x|^2  L_k^{(0)}(|y|^2) e^{- \frac{|x|^2}{2}}\nonumber\\
&=& 2k h_k - 2k h_{k-1} + (k+1) h_{k+1} - (2k +1) h_k + k h_{k-1}\nonumber\\
&=& (k+1) h_{k+1}  - h_k - k h_{k-1}.  \label{bemol}
\end{eqnarray}
We will also need the following formulae, whose proof is postponed in the Section \ref{hermint}: 

\begin{proposition}[Inner products of Hermite functions]
\label{prop21}
We have 
\begin{equation}
\label{h100}
(h_1,h_0 h_0)_{L^2} = \int_{\R^2} h_1(y) h_0(y) h_0(y)  \dd y = \frac{2}{9 \sqrt{\pi}}, 
\end{equation}
and 
\begin{equation}
\label{1100}
\Norm{h_1 h_0}{L^2}^2 = 
\int_{\R^2} h_1(y) h_1(y) h_0(y) h_0(y)  \dd y = \frac{1 }{4 \pi}. 
\end{equation}
\end{proposition}
\subsection{Functions spaces}

We define the space associated with the Harmonic oscillator
$$
\Hc^r = \{ u \in L^2 \, | \,H^{r/2} f \in L^2\}, \quad r \geq 0. 
$$
equipped with the norm 
$$
\Norm{f}{\Hc^r} := \Norm{H^{r/2}f}{L^2}. 
$$
We know (see for instance \cite[Proposition 1.6.6]{Hel84} or \cite[Lemma 2.4]{YZ04})
that on this space the following norms are equivalent: for all $r$ there exist positive constants $c_r$ and $C_r$ such that 
\begin{equation}
\label{normeHr}
c_r \Norm{f}{\Hc^r} \leq  \Norm{f}{H^r} + \Norm{\langle x \rangle^{r} f}{L^2} \leq C_r \Norm{f}{\Hc^r}. 
\end{equation}
Moreover, for $r > 1$ in 2D, $\Hc^r$ is an algebra: there exists $C_r$ such that 
\begin{equation}
\label{algebra}
\forall\, f,g \in \Hc^r \quad \Norm{fg}{\Hc^r} \leq C_r \Norm{f}{\Hc^r} \Norm{f}{\Hc^r}. 
\end{equation}
The space $\Hc^r$ can be described in terms of the coefficients of $f$ in the basis of special Hermite functions $\{\varphi_{n,m}, n \geq 0, -n \leq m \leq n , n + m \mbox{ even }\}$ which is an normalized Hilbertian basis of $L^2(\R^2)$ satisfying (see \cite[Proposition 4.1]{GHT16})
$$
H \varphi_{n,m} = 2 (n +1) \varphi_{n,m} ,\quad L \varphi_{n,m} = m \varphi_{n,m}, 
$$
where $L = i x \times \nabla$ the angular momentum operator. Then every function of $\Hc^r$ expands into 
$$
f = \sum_{n = 0}^{+\infty} \sum_{m = -n}^n c_{n,m} \varphi_{n,m}\quad\mbox{ with }
\quad 
\Norm{f}{\Hc^r}^2 =  \sum_{n = 0}^{+\infty} \sum_{m = -n}^n (2n + 2)^r |c_{n,m}|^2. 
$$
Now {\em radial functions} are $f$ such that $c_{n,m} = ( f, \varphi_{n,m})_{L^2} = 0$ when $m \neq 0$. Then we have $\varphi_{2k,0} = (-1)^k h_k$ defined in \eqref{basishn}. We thus define 
$$
\Hc^r_\rad = \{ \exists \, c_n \in \C^\N\, |\, f = \sum_{n = 0}^\infty c_n h_n \in \mathcal{H}^r \}, 
$$
and for $f = \sum_{n = 0}^\infty c_n h_n \in \Hc^r_\rad $, we have 
$$
\Norm{f}{\Hc^r}^2 = \sum_{n = 0}^{+\infty}  (4n + 2)^r |c_{n}|^2. 
$$
Moreover, for some constant $c_r$ and $C_r$ we have 
$$
 c_r \Norm{f}{\Hc^r}^2 \leq \sum_{n = 0}^\infty \langle n \rangle^{r} |c_n|^2 \leq C_r \Norm{f}{\Hc^r}^2. 
$$
For all $t\in \R$, we can define action of the semi-group $e^{i t H}$ by the formula 
$$
e^{i t H } f =  \sum_{n = 0}^\infty e^{i t \lambda_n }c_n h_n, 
$$
where the $\lambda_n = 4 n + 2$ are the eigenvalues of $H$ on radial functions (see \eqref{lambdan}). 
Note that we have 
\begin{equation}
\label{eqflow}
\forall\, t \in \R,\quad 
\Norm{e^{it H} f}{\Hc^r} = \Norm{f}{\Hc^r}. 
\end{equation}
The operator $f \mapsto |x|^2 f$ acts on functions in $\Hc^r$, and
estimate \eqref{eqy2} implies the following estimate: 
\begin{lemma}
For $r \geq 0$, there exists $C_r$ such that for all $f \in \Hc_{\rad}^{r+1}$, we have 
\begin{equation}
\label{y2op}
\Norm{|x|^2 f}{\Hc^r} \leq C_r \Norm{f}{\Hc^{r+1}}. 
\end{equation}
\end{lemma}
\begin{proof}
Let $f = \sum_{n\geq 0} c_h h_n \in \Hc_{\rad}^{r+1}$. We have by using \eqref{eqy2}
$$
|x|^2 f = \sum_{n \geq 0 } c_n (-(n+1)   h_{n+1}  + (2n +1) h_n - n h_{n-1}) = \sum_{n \geq 0} d_n h_n 
$$
with (defining $c_{n} = 0$ for $n < 0$), 
$$
d_n = - n c_{n-1} + (2 n +1) c_n - (n+1) c_{n+1}
$$
Hence we have 
$$
\Norm{|x|^2 f}{\Hc^r}^2 = \sum_{n \geq 0} \langle n \rangle^r |d_n|^2 \leq 
C \sum_{n \geq 0} \langle n \rangle^r ( n+1 ) ( |c_{n-1}|^2 + |c_n|^2 + |c_{n+1}|^2) 
$$
for some numerical constant $C$, and we easily deduce the result. 
\end{proof}

For two operators $A$ and $B$ acting on functions $f \in \Hc^r$, we set $[A,B]f = (AB - BA)f$. 
\begin{lemma}
\label{commutator}
For $r \geq 0$ there exists a constant $C_r$ such that for all  $f = \sum_{n \geq 0}c_n h_n \in \Hc^{r}_{\rad}$, 
\begin{equation}
\label{eqcom}
\Norm{[H^{r/2},|x|^2] f}{L^2} \leq C_r \Norm{f}{\Hc^{r}}. 
\end{equation}
\end{lemma}
\begin{proof}
From \eqref{eqy2} and \eqref{lambdan} we have 
\begin{eqnarray*}
|x|^2 H^{r/2} h_k(x) &=& (4k +2)^{r/2} |y|^2 h_k \\
&=&  (4k +2)^{r/2} (-(k+1) h_{k+1}(x)  + (2k +1) h_k(x) - k h_{k-1}(x)) ,
\end{eqnarray*}
and 
$$
H^{r/2} |x|^2 h_k(x) = -(k+1)(4k +6)^{r/2} h_{k+1}(x)  + (2k +1)(4k +2)^{r/2} h_k(x) - (4 k - 2)^{r/2} k h_{k-1}(x).
$$
Hence as for all $k,\ell \geq 0$, $r \geq  0$ and some constant $C_r$ independent of $k$ and $\ell$,   we have 
\begin{equation}
\label{eouais}
|k^{r/2} - \ell^{r/2} | \leq C_r |k - \ell| ( k^{r/2 - 1} + \ell^{r/2 - 1})
\end{equation}
we deduce that 
$$
(H^{r/2} |x|^2 - |x|^2 H^{r/2}) h_k = \alpha_{k} h_{k+1} + \mu_k h_k + \beta_k h_{k-1}
$$ 
with $|\mu_k| + |\alpha_k| + |\beta_k| \leq C_r \langle k \rangle^{r/2}$ for some constant $C_r$ independent on $k$. Hence if $v = \sum_{n \geq 1} c_n h_n \in \Hc^r_{\rad}$, we have $[H^{r/2},|x|^2] v = \sum_{n} d_n h_n$ with 
$$
d_n = c_{n-1}\alpha_{n-1} + c_n \mu_n + c_{n+1} \beta_{n+1}
$$ 
and hence 
$$
|d_n|^2 \leq C_r \langle n \rangle^{r} \big( |c_{n-1}|^2 + |c_{n}|^2 + |c_{n+1}|^2\big) 
$$
for some constant $C_r$ independent of $n$, 
from which we easily deduce \eqref{eqcom}. 
\end{proof}
\subsection{CR operator}
The CR trilinear operator is given by 
$$
(f_1,f_2,f_3) \mapsto \Tc(f_1,f_2,f_3)(z) = \int_{\R^2} \int_{\R} f_1(x + z) \overline{f_2(x + \lambda x^{\perp} + z )} f_3(\lambda x^{\perp} + z)  \, \dd \lambda \dd x. 
$$
where for $x = (x_1,x_2)\in\R^2$, we set $x^\perp = (-x_2,x_1)$. 
With this trilinear operator is associated the energy 
\begin{equation}
\label{energyCR}
\Ec(f_1,f_2,f_3,f_4) = \int_{\R^2}\int_{\R^2} \int_{\R} f_1(x + z) f_2(\lambda x^{\perp} + z) \overline{f_3(x + \lambda x^{\perp} + z )} \, \overline{f_4(z)} \, \dd \lambda \dd x \dd z
\end{equation}
and the (CR) equation introduced in \cite{FGH16}
\begin{equation}
\tag{CR} i \partial_t f = \Tc(f,f,f). 
\end{equation}
We recall some properties of the CR operator that can be found in \cite{FGH16} and \cite{GHT16}.  First it is invariant by Fourier transform: 
$$
\Fc ( \mathcal{T}(f_1,f_2,f_3)  ) = \Tc( \widehat f_1,\widehat f_2,\widehat f_3) \quad \mbox{and} \quad 
\Ec(f_1,f_2,f_3,f_4) = \Ec(\widehat f_1,\widehat f_2,\widehat f_3, \widehat{f_4}) 
$$
Moreover, this operator has many symmetries that are summarized in Table \ref{table1}. In this table, $Q$ denotes a self adjoint operator {\em commuting} with $\Tc$ in the sense of Lemma 2.4 in \cite{GHT16}, {\em i.e.}
$$
Q( \Tc(f_1,f_2,f_3) ) =  \Tc( Q f_1,f_2,f_3) - \Tc( f_1, Q f_2, f_3) + \Tc(f_1,f_2,Q f_3), 
$$
as soon as $f_1,f_2$ and $f_3$ are in the domain of $Q$. Then for all $\lambda \in \R$, we have 
$$
e^{i \lambda Q}\Tc(f_1,f_2,f_3)  = \Tc( e^{i \lambda Q}f_1,e^{i \lambda Q}f_2,e^{i \lambda Q}f_3). 
$$
 With such an operator is associated an invariant $\int_{\R^2} (Qu)(x) u(x) \dd x
$ of the (CR) equation. We also use the notation $R_\theta(x_1,x_2) = (x_1 \cos\theta - x_2 \sin \theta, x_2 \sin\theta + x_1 \cos \theta)$ the rotation of angle $\theta$. 

Finally, the operator $\Tc$ is trilinear in $\Hc^r$ as can be immediatly seen from formula (2.4) in \cite{GHT16} as well as  \cite[Proposition 7.1]{FGH16}: we have for $r \geq 0$
\begin{equation}
\label{contiCR}
\Norm{\Tc(f_1,f_2,f_3)}{\Hc^r} \leq C_r \Norm{f_1}{\Hc^r} \Norm{f_2}{\Hc^r} \Norm{f_3}{\Hc^r}. 
\end{equation}
for some consant $C_r$ independent of $f_1$, $f_2$ and $f_3$. 
\renewcommand{\arraystretch}{1.2}
\begin{table}
\label{table1}
\begin{center}
\begin{tabular}{|c|c|c|}
\hline
Operator $Q$ & Conserved quantity & Corresponding symmetry \\
commuting with $\Tc$  &  $\int (Q u)  \bar u $
& $u \mapsto e^{i \lambda Q} u$ \\
\hline
$1$ & $\int |u|^2$ & $u \mapsto e^{i \lambda } u$\\
\hline
$x_1$ & $\int x_1 |u|^2$ & $u \mapsto e^{i \lambda x_1} u$\\ 
\hline
$x_2$ & $\int x_2 |u|^2$ & $u \mapsto e^{i \lambda x_2} u$\\ 
\hline
$|x|^2$ & $\int |x|^2 |u|^2$ & $u \mapsto e^{i \lambda |x|^2} u$\\ 
\hline
$i \partial_{x_1}$ & $\int ( i \partial_{x_1} u) \bar u $ & $u \mapsto u(\, \cdot \, + \lambda e_1) $\\ 
\hline
$i \partial_{x_2}$ & $\int ( i \partial_{x_2} u) \bar u$ & $u \mapsto u(\, \cdot \, + \lambda e_2)$\\ 
\hline
$\Delta$ & $\int |\nabla u|^2$ & $u \mapsto e^{i \lambda \Delta} u$\\ 
\hline
$H$ & $\int (Hu) \bar u$ & $u \mapsto e^{i \lambda H} u$\\ 
\hline
$L = i ( x \times \nabla)  $ & $\int (Lu) \bar u$ & $u \mapsto u \circ R_\lambda$\\ 
\hline
$i ( x \cdot \nabla + 1)$ & $\int i (x \cdot \nabla + 1) u \bar u$ & $u \mapsto \lambda u (\lambda \, \cdot \,)$\\ 
\hline
\end{tabular}
\end{center}
\caption{Symmetries of the CR equation}
\end{table}

Finally, for some function $f \in \Hc^r$ and $r\geq 0$, we define the operator 
\begin{equation}
\label{Tfu}
\Tc[f] u := \Tc(f,f,u). 
\end{equation}
\subsection{CR operator on radial functions} 
Following again \cite[Section 6.6]{GHT16}, if $h_n$ denote the Laguerre-Hermite functions \eqref{basishn} then we have  for all $n_1,n_2$ and $n_3 \in \N$, 
\begin{equation}
\label{CRhn}
\Tc( h_{n_1}, h_{n_2}, h_{n_3}) = \chi_{n_1 n_2 n_3 n_4} h_{n_4}, \quad n_4 = n_1 - n_2 + n_3. 
\end{equation}
where
\begin{equation}
\label{piaz}
\chi_{n_1 n_2 n_3 n_4} = \pi^2 \int_{\R^2} h_{n_1}(x)h_{n_2}(x)h_{n_3}(x)h_{n_4}(x) \dd x. 
\end{equation}
Using \eqref{1100} we obtain 
$\chi_{1100} = \pi^2 \Norm{h_1 h_0}{L^2}^2 = \frac{\pi}{4}$. 


\section{Modulation and the pseudo conformal symmetry}


We set up in this section the basic algebraic fact and energy estimates associated to modulation of the {\em unperturbed} linear flow \eqref{defH}. The essential algebraic fact is the existence of an explicit pseudo-conformal symmetry which will generate the modulated bubbles, and more importantly the leading order finite dimensional dynamical system to be perturbed in a {\em resonant way}.


\subsection{Commutators formulae}


First, for $u \in \Hc^r$ and $N > 0$, we define the operator 
\begin{equation}
\label{ScN}
(\Sc_{N} u)(x) = \frac{1}{N} u( \frac{x}{N}). 
\end{equation}
Note that in dimension 2, we have 
\begin{equation}
\Norm{\Sc_N u }{L^2} = \Norm{u}{L^2} \quad \mbox{and}\quad \widehat{\Sc_N u}= \Sc_{\frac1N} u. 
\end{equation}
We will also modulate by using the operators $e^{i m |x|^2}$ and $e^{i c \Delta}$ for real numbers $c$ and $m$. Note that all these transformation preserve the radial symmetry. 
We now collect here the following commutator relations: 
\begin{lemma}[Commutators]
For $c$ and $m$ real numbers, 
we have the relations
\begin{align}
&e^{ - i m |x|^2  }  \Delta e^{  i m |x|^2} =   \Delta  - 4 m^2 |x|^2 +   4 i m ( 1 + \Lambda), \label{C1}\\
&e^{- i c \Delta} |x |^2  e^{i c \Delta} =  |x|^2  - 4 c^2 \Delta -   4 i c ( 1 + \Lambda), \label{C2} \\
&  e^{ - i m |x|^2} ( 1 + \Lambda) e^{i  m |x|^2} =  (1 + \Lambda)  + 2 i m  |x|^2,\label{C3}\\
&e^{- i c \Delta} ( 1 + \Lambda) e^{i c \Delta} = (1 + \Lambda) - 2 i c  \Delta. \label{C4}
\end{align}
\end{lemma}
\begin{proof} We have 
\begin{equation}
\label{nab}
\nabla_x (e^{  i m |x|^2 } u) = e^{  i m |x|^2 } (\nabla u + 2i m x u), 
\end{equation}
and hence 
\begin{eqnarray*}
\nabla_x \cdot \nabla_x (e^{  i m |x|^2 } u)  &=& e^{  i m |x|^2} (2 i m x) \cdot (\nabla u   +2i m x u)  \\ &&+ e^{  i m |x|^2 } ( \Delta u  + 2i m \nabla \cdot ( x u))\\
&=&e^{  i m |x|^2 } \big(  2 i m x \cdot \nabla - 4 m^2 |x|^2    + \Delta  + 4im  + 2 i m x \cdot \nabla \big) u, 
\end{eqnarray*}
Which yields the first equation. 
In terms of Fourier transform, we have 
$$
(\widehat{(1 + \Lambda)  u })(\xi)  =  \widehat{u}(\xi) - \nabla_\xi \cdot ( \xi \widehat u (\xi) ) = - ((1  + \Lambda) \widehat u)(\xi). 
$$
which we can write $\Fc^{-1} (1+ \Lambda) \Fc = - (1 + \Lambda)$. 
Now using that $\Fc^{-1} |\xi|^2 \Fc = - \Delta$, and $\Fc^{-1} e^{  i m \Delta } \Fc = e^{  -i m  |\xi|^2 }$,  the first relation implies
$$
- e^{  i m \Delta }  |\xi|^2 e^{  - i m \Delta } =  - |\xi|^2 +  4 m^2  \Delta  -   4 i m ( 1 + \Lambda),
$$
which yields the second relation, after taking $m = - c$. 
Now from \eqref{nab}, we obtain 
$$
\Lambda (e^{  i m |\xi|^2} u) = e^{  i m |\xi|^2} (\Lambda u + 2i m |\xi|^2 u), 
$$
and hence the third line. The fourth is obtain by Fourier transform.


\subsection{Energy estimates through modulation}


We take $c$, $m$ and $N$ as in the previous section and we are interested in estimating the Sobolev norms of $u = \Sc_N e^{i m |x|^2} e^{i c \Delta } v$ with respect to the norms of $v$. 
We will need the following lemma, whose proof can be found for instance in \cite{DR}. 
\begin{lemma}[Fourier transform of Gaussians]
Let $u = u_1 + i u_2 \in \C$  with $\mathrm{Re}\, u \geq 0$. Then we have 
\begin{equation}
\label{foufoug}
\mathcal{F}\Big( e^{- u| \, \cdot \, |^2} \Big) (\xi) =   \frac{1}{2 u} e^{- \frac{\xi^2}{4u}}. 
\end{equation}
\end{lemma}

\begin{proposition}[Energy estimates through modulations]
\label{lemnorms}
Let $r \in \N$. Then there exists $C_r$ such that for all $v \in \Hc^r$ and all real numbers $m$, $c$, $N$, we have 
\begin{equation}
\label{boun}
\Norm{\Sc_N     e^{i m |x|^2} e^{i c \Delta}  v }{\Hc^r} \leq C_r ( \langle m \rangle^{r} + \langle c \rangle^r) \max\big(N,\frac{1}{N}\big)^r \Norm{v}{\Hc^r}. 
\end{equation}
If moreover $v(y) = h_0(y)$, then with 
$$
\eta = 1 + 2ic \quad \mbox{and} \quad z = 1 + 4cm - 2 i m. 
$$
we have 
\begin{eqnarray*}
\Norm{ |x|^{2r} \Sc_N     e^{i m |y|^2} e^{i c \Delta}  h_0}{L^2}^2 &=&  N^{2r } |\eta|^{2r}  \int_{\R^2}   |x|^{2r} |h_0(x)|^{2}\dd x \\
 \Norm{ |\nabla|^{2r} \Sc_N     e^{i m |y|^2} e^{i c \Delta}  h_0}{L^2}^2 &=&   \frac{1}{N^{2r}} |z|^{2r}    \int_{\R^2}   |x|^{2r} |h_0(x)|^{2}\dd x
\end{eqnarray*}
and in particular
\begin{equation}
\label{normegaussienne}
\Norm{\Sc_N     e^{i m |y|^2} e^{i c \Delta}  h_0}{\Hc^1}^2 =  N^{2 }\big( 1 + 4 c^2)  + \frac{1}{N^{2}} ((1 + 4cm)^2 + 4m^2\big).  
\end{equation}
\end{proposition}
\begin{proof}
By homogeneity,
$$
\Norm{\Sc_N v}{H^r}^2 \leq C_r \max(1,\frac{1}{N^2})^r \Norm{v}{H^r}^2 \quad\mbox{and} \quad 
\Norm{\langle x \rangle^r \Sc_N v }{L^2}^2 \leq C_r  \max(1,N^2)^r \Norm{\langle y \rangle^r v}{L^2}^2 
$$
and hence using \eqref{normeHr}, 
$$
\Norm{\Sc_N v}{\Hc^r}^2 \leq C_r \max(N^2,\frac{1}{N^2})^r \Norm{v}{\Hc^r}^2. 
$$
Moreover, we have $\Norm{\langle y \rangle^r e^{i m |x|^2} v}{L^2} =  \Norm{\langle y \rangle^r v}{L^2}$ and as for $i = 1,2$, 
$$
\partial_{x_i} e^{i m |x|^2} v  = e^{ i m |x|^2}(  2 i m x_i  v + \partial_{x_i} v), 
$$
we have the estimate 
$$
\Norm{\langle\nabla\rangle  e^{i m |x|^2} v }{L^2}  \leq C \langle m \rangle  \Norm{  \langle  x \rangle v }{L^2} + \Norm{\langle \nabla\rangle v}{L^2}, 
$$
By iterating this estimate, we obtain 
$$
\Norm{e^{i m |x|^2} v }{\Hc^r}  \leq C \langle m \rangle^r \Norm{v}{\Hc^r}. 
$$
and hence after Fourier transform $\Norm{e^{i c \Delta} v }{\Hc^r}  \leq C \langle c \rangle^r \Norm{v}{\Hc^r}$, which yields \eqref{boun}.  

To calculate the norm of the modulated Gaussian, using \eqref{foufoug} we calculate that 
\begin{eqnarray*}
e^{i m |x|^2 } e^{i c \Delta} h_0 &=& \frac{1}{\sqrt{\pi} } e^{i m |x|^2 } \Fc ( e^{ - (\frac{1}{2}+ i c) |\, \cdot\, |^2 }) 
= \frac{1}{\eta \sqrt{\pi}} e^{i m |x|^2 }  e^{- \frac{1}{2 \eta} |x|^2 } = 
\frac{1}{\eta \sqrt{\pi}} e^{- \frac{z}{2 \eta} |x|^2 }
\end{eqnarray*}
We thus have 
$$
\Sc_N e^{i m |x|^2 } e^{i c \Delta} h_0 = \frac{1}{N \eta \sqrt{\pi}} e^{- \frac{z}{2 \eta N^2} |x|^2 } \quad
\mbox{and} \quad 
\Fc \Sc_N e^{i m |x|^2 } e^{i c \Delta} h_0  =  \frac{N}{z \sqrt{\pi}} e^{- \frac{N^2\eta}{2 z} |\xi|^2 }. 
$$
Note that we have 
$\mathrm{Re}(z \overline{\eta}) = 1$. Hence
\begin{eqnarray*}
\int_{\R^2} |x|^{2r} |\Sc_N e^{i m |x|^2 } e^{i c \Delta} h_0|^2 \dd x &=& 
\frac{1}{\pi} \int_{\R^2}\frac{|x|^{2r}}{N^2 |\eta|^2} e^{- \frac{1}{|\eta|^2 N^2} |x|^2 } \dd x \\&=&
N^{2r} |\eta|^{2r} \frac{1}{\pi}  \int_{\R^2}|x|^{2r} e^{-  |x|^2 } \dd x, 
\end{eqnarray*}
which yields the first estimate. The second one is obtained by Fourier transform. 
\end{proof}

\subsection{Modulation equation} 

We consider the equation 
\begin{equation}
\label{eqH}
i \partial_t u = - \Delta u + |x|^2 u. 
\end{equation}

\begin{proposition}[Modulated pseudo conformal symmetry]
Let $L(t) > 0$, $\gamma(t)$ and $b(t)$ be real functions defined on $\R_+$. We set 
\begin{equation}
\label{prince}
u  = e^{i \gamma} \Sc_{L} w, \quad w(t,y) = e^{- i \frac{b|y|^2}{4}} v(t,y) 
\qquad \mbox{and}\qquad 
\displaystyle \frac{\dd s}{\dd t} = \displaystyle\frac{1}{L^2}. 
\end{equation}
Assume that the function $t \mapsto s(t)$  is invertible from $\R$ to itself. Then $u(t,x)$ solves \eqref{eqH} if and only $v(s,y)$ solves the equation 
\begin{equation}
\label{eqH2}
i \partial_s v + \Delta v - \gamma_s v + \Big( - L^4+  \frac{b_s}{4} - \frac{b^2}{4}  -   \frac{L_s}{L}\frac{b}{2}  \Big) |y|^2 v  - i \Big(  \frac{L_s}{L} + b \Big)\Big(1 + \Lambda \Big) v = 0,
\end{equation}
where $L_s = \frac{\dd}{\dd s}( L(t(s)) )$ and similar definitions for $b_s$ and $\gamma_s$. 
\end{proposition}

\begin{proof}
We compute
\begin{eqnarray*}
i \partial_t u = 
i \partial_t e^{i \gamma} \Sc_L w  &=& i \partial_t  \Big( \frac{1}{L} e^{i \gamma} w(t,\frac{x}{L}) \Big) =  e^{i \gamma}\Sc_L \Big( - i \frac{L_t}{L} (  1+ \Lambda ) w + i \partial_t w - \gamma_t w \Big) \\
&=& \frac{e^{i \gamma}}{L^{2}} \Sc_L \Big( - i \frac{L_s}{L} (  1+ \Lambda ) w + i \partial_s w - \gamma_s w \Big),  
\end{eqnarray*}
with the notation $L_t = \partial_t L = \frac{1}{L^2} L_s$ and similar notations for the derivatives of $\gamma$ and $b$. 
Hence $u = e^{i \gamma} \Sc_L w$ is solution of \eqref{eqH} if and only if 
$$
\frac{1}{L^{2}}\Big( - i \frac{L_s}{L} (  1+ \Lambda ) w + i \partial_s w - \gamma_s w \Big) = 
- \Sc_L^{-1} \Delta \Sc_L w + \Sc_L^{-1} |x|^2 \Sc_L w  
$$
and we obtain the equation 
$$
 - i  \frac{L_s}{L} \Big(1+ \Lambda \Big)w + i \partial_s w - \gamma_s w 
+  \Delta_y w (s,y) - L^4  |y|^2 w = 0. 
$$
Now as $w (s,y) = e^{- i \frac{b |y|^2}{4}} v(s,y)$, we have 
$$
i \partial_s w = e^{- i \frac{b|y|^2}{4}} (  \frac{b_s}{4} |y|^2 v + i \partial_s v).  
$$
Hence we obtain the equation 
$$
i \partial_s v +  \frac{b_s}{4} |y|^2 v  - \gamma_s v - i\frac{L_s}{L} e^{i\frac{b|y|^2}{4}  }( 1 + \Lambda) e^{- i \frac{b|y|^2}{4}}
+ e^{i\frac{b|y|^2}{4}  } \Delta  e^{ - i\frac{b|y|^2}{4}  } v - L^4 |y|^2 v = 0, 
$$
and we obtain the result with \eqref{C1} and \eqref{C3} with $m = - \frac{b}{4}$. 
\end{proof}

\begin{remark}
In the following, we will always be in situations where $s \mapsto t(s)$ is invertible. 
We will thus write by a slight abuse of notation $(b(s),L(s),\gamma(s))$ for the functions $b((t(s))$, $L(t(s))$,  and $\gamma(t(s))$. 
\end{remark}


\section{Hamiltonian structures of the modulation equations}


From \eqref{eqH2}, the explicit choice
\begin{equation}
\label{crux}
\left|
\begin{array}{l}
\displaystyle  - L^4   +  \frac{b_s}{4} - \frac{b^2}{4}  -   \frac{L_s}{L}\frac{b}{2}  = - 1 \\[2ex]
\displaystyle   \frac{L_s}{L} + b = 0
\end{array}
\right.
\end{equation}
maps \eqref{eqH} onto 
$$
i \partial_s v  = H v  + \gamma_s v. 
$$
for which $v =  h_k$ and $\gamma_s = - \lambda_k$ provide stationnary solutions. The dynamical system \eqref{crux} can be integrated explicitiely and the obtained transformation \eqref{prince} is nothing but the classical pseudo conformal symmetry (or Lens transform) of \eqref{defH}. Our aim in this section is to recall  the classical Hamiltonian setting to integrate \eqref{crux} which prepares for the perturbative analysis performed in section \ref{sectionmodualtib}.


\subsection{Darboux-Lie transform}


The dynamical system \eqref{crux} can be written 
$$
\frac{\dd}{\dd s}
\begin{pmatrix}
L \\ b
\end{pmatrix}
= \begin{pmatrix}
- b L \\ -b^2
 - 4 + 4  L^4 \end{pmatrix}
 = 2L^3 \begin{pmatrix}
0 & -1 \\ 1 & 0
\end{pmatrix}
\begin{pmatrix}
\partial_L E \\ \partial_b E
\end{pmatrix}
$$
where 
$$
E = \frac{1}{L^2}(\frac{b^2}{4} + 1) +   L^2. 
$$

We want to write the previous system in a canonical Hamiltonian form (such a change of coordinates is called Darboux-Lie transformation). 

\begin{lemma}
\label{lem30}
 Let $(L,b) \in (0,+\infty) \times \R$, $H(L,b)$ be given function, and 
let a non canonical Hamiltonian system 
$$
L_s = - 2L^3 \partial_b H(L,b) \quad \mbox{and} \quad b_s = 2L^3 \partial_L H(L,b) 
$$
be given. Then the change of variable $(L,b) \mapsto (\ell,b)$ where 
$$
\ell = \frac{1}{4L^2}
$$
transform the system into a canonical Hamiltonian system of the form 
$$
\frac{\dd}{\dd s}
\begin{pmatrix}
b \\ \ell
\end{pmatrix}
= \begin{pmatrix}
0 & -1 \\ 1 & 0
\end{pmatrix}
\begin{pmatrix}
\partial_b K \\ \partial_\ell K
\end{pmatrix}
\quad \mbox{where} \quad K(b,\ell) = H(L,b).   
$$
\end{lemma}
\begin{proof}
With $K(\ell,b) = H(\frac{1}{2 \sqrt{\ell}},b)$
we calculate that 
$$
\ell_s  = - \frac{L_s}{2L^3} = \partial_b H(L,b) = \partial_{b} K(\ell,b)
$$
and moreover 
$$
\partial_\ell K(\ell,b) = - \frac{1}{4 \ell^{3/2}} \partial_L H(\frac{1}{2 \sqrt{\ell}},b) = - 2L^3 
\partial_L H(L,b) = - b_s, 
$$
which shows the result. 
\end{proof}

\subsection{Action-angles variables}

In the canonical variables $(b,\ell) \in \R \times (0,+\infty) $, the Hamiltonian associated with the system \eqref{crux} is given by 
$$
E(b,\ell) = \ell b^2 + 4 \ell + \frac{1}{4\ell} >2, 
$$
where with a slight abuse of notation, we note $E$ the Hamiltonian in variables $(\ell,b)$ as in variables $(L,b)$. 
The system \eqref{crux} is thus equivalent to the system 
\begin{equation}
\label{cruxell}
\left|
\begin{array}{rcll}
b_s &=& - \partial_{\ell} E & = - b^2 - 4 + \frac{1}{4 \ell^2}\\[2ex]
\ell_s &=& \partial_{b}E &= 2 \ell b
\end{array}
\right.
\end{equation}
\begin{proposition}
\label{propAV}
There exists a symplectic change of variable $(b,\ell) \mapsto (a,\theta)$ from the set 
$ \R\times (0,+\infty)$ to $(\frac12,+\infty) \times \T$ such that 
$$
E (b,\ell) = 4a, \quad \mbox{ so that } \quad \theta_s = 4, \quad a_s = 0,
$$
and the flow in variable $(\theta,a)$ is given by $a(s) = a(0)$ and $\theta(s) = \theta(0) + 4s$. 
Moreover, we have the explicit formulae 
\begin{equation}
\label{changeco}
\begin{array}{rcl}
\ell &=& \displaystyle \frac{1}{4} \left( 2a  - \sqrt{4a^2 - 1}\cos(\theta)\right) = \frac{1}{4L^2} \quad \mbox{and}\\[2ex]
b \ell &=& \sqrt{4a^2 - 1} \sin(\theta) = \displaystyle \frac{b}{4L^2}. 
\end{array}
\end{equation}
Moreover, we can expand $L^2(a,\theta)$ as follows:  
\begin{equation}
\label{fish}
L^2 (\theta,a) = 1 + 2 \sum_{n > 0}  \left( \frac{2a - 1}{2a + 1} \right)^{\frac{n}{2}} \cos(n \theta) 
\end{equation}
\end{proposition}
\begin{proof}
We use the method of generating functions with $b$ as impulse variable. We write on the set $\{ b > 0\}$ describing half a period,  
\begin{equation}
\label{bell}
b = \sqrt{- 4 + \frac{E}{\ell} - \frac{1}{4\ell^2}} = \partial_\ell S(E,\ell)
\end{equation}
where for $E > 2$ and $\ell >0$, 
$$
S(E,\ell) = \int_{\ell_0}^\ell \sqrt{- 4 + \frac{E}{z} - \frac{1}{4z^2}} \,  \dd z =  \int_{\ell_0}^\ell
\frac{1}{2z} \sqrt{- 16z^2  + 4Ez - 1} \dd z. 
$$
Note that here, 
$$
\ell \in \Big[\frac{1}{8}(E - \sqrt{E^2 - 4}), \frac{1}{8}(E + \sqrt{E^2 - 4}) \Big]. 
$$
Now by construction, the change of variable $(b,L) \mapsto (E,\psi)$ is symplectic, with 
$$
\psi =  \partial_E S(E,\ell) =\int_{\ell_0}^\ell
\frac{1}{\sqrt{- 16z^2  + 4Ez - 1}} \dd z =   \int_{\ell_0}^\ell
\frac{1}{\sqrt{ \frac{E^2}{4} - 1 - ( 4 z - \frac{E}{2})^2  }} \dd z. 
$$
Moreover, we have in view of \eqref{cruxell} and \eqref{bell}
$$
\frac{\dd}{\dd s} \psi(s) = \frac{\ell_s}{\sqrt{- 16\ell^2  + 4E\ell - 1}} = \frac{ b \ell }{\sqrt{- 4\ell^2  + E\ell - \frac{1}{4}}} = 1. 
$$
Now we have  
$$
\psi = \frac{1}{\sqrt{\frac{E^2}{4} - 1}} \int_{\ell_0}^\ell
\frac{1}{\sqrt{ 1 -  ( \frac{4 z - \frac{E}{2}}{\sqrt{\frac{E^2}{4} - 1}})^2  }} \dd z
= \frac{1}{\sqrt{\frac{E^2}{4} - 1}} \int_{\ell_0 - \frac{E}{8}}^{\ell - \frac{E}{8}}
\frac{1}{\sqrt{ 1 -  ( \frac{4 z }{\sqrt{\frac{E^2}{4} - 1}})^2  }} \dd z,
$$
or
$$
\psi = \frac{1}{4} \int_{\frac{4(\ell_0 - \frac{E}{8})}{\sqrt{\frac{E^2}{4} - 1}}}^{\frac{4(\ell - \frac{E}{8})}{\sqrt{\frac{E^2}{4} - 1}}}
\frac{1}{\sqrt{ 1 -  z^2  }} \dd z = \frac{1}{4}\arcsin \frac{4(\ell - \frac{E}{8})}{\sqrt{\frac{E^2}{4} - 1}} + \frac{\pi}{8} \in [0, \frac{\pi}{4}].  
$$
by taking $\ell_0 = \frac{1}{8}(E - \sqrt{E^2 - 4})$ so that $\frac{4(\ell_0 - \frac{E}{8})}{\sqrt{\frac{E^2}{4} - 1}} = -1$. This change of variable describes half-a period. 
In order to obtain action-angle  we set $(a,\theta) = (  E/4,4\psi) \in  (\frac{1}{2},+\infty) \times \T$ to obtain action angle with $\theta \in [0,2\pi]$ on a full period, and a Hamiltonian $E(a,\theta) = 4a $. 
We thus have $\theta_s = \partial_a E(a,\theta) = 4$, $a_s =  - \partial_{\theta} E(a,\theta) = 0$ and 
$$
\theta =  \frac{\pi}{2} + \arcsin \frac{4(\ell - \frac{a}{2})}{\sqrt{4 a^2 - 1}} 
$$
and hence 
$$
\frac{4(\ell - \frac{a}{2})}{\sqrt{4a^2 - 1}} = \sin( \theta - \frac{\pi}{2} ) = - \cos(\theta). 
$$
and thus
$$
\ell = \frac{1}{4} \left( 2a  - \sqrt{4a^2 - 1}\cos(\theta)\right)
$$
and
$$
b = \frac{1}{2\ell} \sqrt{- 16 \ell^2 + 4 E \ell  - 1} = 
\frac{1}{2\ell} \sqrt{ \frac{E^2}{4} - 1 - ( 4 \ell - \frac{E}{2})^2  } = \frac{1}{\ell}\Big(\sqrt{4a^2 - 1} \Big)\sin(\theta), 
$$
which is positive for $\theta \in [0,\pi]$. 
In particular, we have 
$$
b \ell = \sqrt{4a^2 - 1} \sin(\theta), 
$$
which shows \eqref{changeco}. 

To prove \eqref{fish}
we can expand in Fourier series. We have 
$$
L^2 = \frac{1}{2a  - \sqrt{4a^2 - 1}\cos(\theta)} = \frac{1}{2 a}\left( \frac{1 }{1 - \sqrt{1 - \frac{1}{4a^2}} \cos(\theta)}\right)
$$
We recall the formula for the Poisson kernel, for $r \in (0,1)$, 
\begin{eqnarray*}
\sum_{n \in \Z } r^{|n|} e^{i n\theta} &=& \frac{1 - r^2}{1 + r^2  -2 r \cos(\theta)} \\ &=& \frac{1 - r^2}{1 + r^2} \left(\frac{1}{ 1- 2 \frac{r}{1+r^2} \cos(\theta)}\right). 
\end{eqnarray*}
To apply the formula, we need to take
$$
2 \frac{r}{1+r^2} = \sqrt{1 - \frac{1}{4a^2}}. 
$$
Setting $\alpha = 2 \frac{r}{1+r^2} $ or $r^2 + 1 - 2 \frac{r}{\alpha} = 0$, this yields to   
\begin{multline*}
r = \frac{1}{\alpha} -
\sqrt{\frac{1}{\alpha^2} - 1} = \frac{1}{\alpha}( 1 - \sqrt{1 - \alpha^2})\\
= \frac{1}{\sqrt{1 - \frac{1}{4a^2}}} ( 1 - \sqrt{1 - 1 + \frac{1}{4a^2}})
 = \frac{2a} {\sqrt{4a^2 - 1}}\left( 1 - \frac{1}{2a}\right).
\end{multline*} 
and hence 
$$
r = \frac{2 a - 1 }{\sqrt{4 a^2 - 1}} = \sqrt{\frac{2a - 1}{2a + 1}}
$$ which is indeed in $(0,1)$. Note that we have 
$$
\frac{1 + r^2}{1 - r^2} = \frac{2a + 1 + 2a - 1}{2a + 1 - 2a + 1} = 2a
$$
This shows \eqref{fish}. 
\end{proof}

\subsection{Resonant bubbles}

With the action-angle variables in hand, we are able to completely solve the system \eqref{crux} and provide solutions to the system \eqref{prince}. 

Taking $\theta = 4s$, we can indeed solve $t$ in terms of $s$ as follows: 
We have 
$$
\frac{\dd t}{\dd s } = L^2 = \frac{1}{4 \ell} =  \frac{1}{2a  - \sqrt{4a^2 - 1}\cos(4s)}
$$ 
Using \eqref{fish}, we can solve to solve the system in time: 
$$
t(s) = \displaystyle   s + \sum_{n >0} \frac{1}{2n} \left( \frac{E - 2}{E+ 2} \right)^{\frac{n}{2}} \sin(4n s).
$$  
We summarize by the formulas in terms of $E$ for the free flow 
$$
\left|
\begin{array}{rcl}
L^2(s) = \frac{1}{4\ell (s)}&=&  \displaystyle\frac{2}{ E   - \cos(4s)\sqrt{E^2 - 4}}   \\[2ex]
b(s) &=& \displaystyle\frac{\sin(4s)}{2 \ell(s)}\sqrt{E^2 - 4} =  \displaystyle\frac{4\sin(4s)\sqrt{E^2 - 4}}{ E   - \cos(4s)\sqrt{E^2 - 4}}  \\[2ex]
t(s) &=& \displaystyle   s + \sum_{n >0} \frac{1}{2n} \left( \frac{E - 2}{E+ 2} \right)^{\frac{n}{2}} \sin(4n s)
\end{array}
\right.
$$
Note that this formula together with the fact that $\frac{\dd t}{\dd s} > 0$ shows that 
shows that $t(s)$ is invertible and these formula with the change of variable \eqref{prince} provide solutions to the free flow which are all oscillating for at the same frequency for all values of $E$.

\section{The resonant trajectory}
\label{sectionmodualtib}


We are now in position to study small perturbations of \eqref{crux} and prove the existence of resonant trajectories for a suitable choice of perturbations. Let us stress that we need in global in time bounds in the presence of highly oscillatory solutions, and these will be provided by the systematic use of action-angle variables and the backwards in time integration method.


\subsection{Perturbed Hamiltonian}


Let us consider the action-angle variables defined in Proposition \eqref{propAV}. The unperturbed system \eqref{crux} is associated with the Hamiltonian $E(a,\theta) = 4a$. Let us consider a time dependent Hamiltonian perturbation of this Hamiltonian of the form 
\begin{equation}
\label{arnold}
H(s,a,\theta) = 4a + P(s,a,\theta). 
\end{equation}
Then the system is given by 
$$
a_s = - \partial_\theta P(s,a,\theta),\quad \mbox{and}\quad \theta_s = 4 + \partial_{a} P(s,a,\theta). 
$$
Now as the change of variable $(b,\ell) \mapsto (a,\theta)$ is symplectic, and with the definition of $\ell$, this dynamical system is equivalent to the following system in coordinates $(L,b)$: 
$$
\left| 
\begin{array}{l}
 - L^4  +  \frac{b_s}{4} - \frac{b^2}{4}  -   \frac{L_s}{L}\frac{b}{2}  = - 1 + 2 L^3 \partial_L P(s,b,L)\\
  \frac{L_s}{L} + b = - 2 L^2 \partial_b P(s,b,L), 
\end{array}
\right. 
$$
where $P(s,b,L) = P(s, a,\theta)$ (see Lemma \ref{lem30}). 
Let $\beta(s)$ be a given function. 
The solution of the equation 
\begin{equation}
\label{perturb0}
\left| 
\begin{array}{l}
 - L^4  +  \frac{b_s}{4} - \frac{b^2}{4}  -   \frac{L_s}{L}\frac{b}{2}  = - 1 - \beta(s)\\
  \frac{L_s}{L} + b = 0, 
\end{array}
\right. 
\end{equation}
is thus the solution of a Hamiltonian of the form $E(b,L) + \frac{\beta(s)}{L^2}$. 
In variable $(a,\theta)$ this Hamiltonian is given by 
\begin{equation}
\label{Hams}
H(s,a,\theta) = 4a + \beta(s) \left( 2a  - \sqrt{4a^2 - 1}\cos(\theta)\right). 
\end{equation}
The dynamical system associated with this Hamiltonian is given by 
\begin{equation}
\label{modbeta}
\left|
\begin{array}{rcll}
\theta_s &=& \displaystyle 4 + 2 \beta(s) -  \beta(s) \frac{4a \cos(\theta)}{\sqrt{4a^2 - 1}} &= \partial_a H(s,a,\theta)\\[2ex]
a_s &=& - \beta(s) \sqrt{4a^2 - 1}\sin(\theta) &= - \partial_{\theta}H(s,a,\theta). 
\end{array}
\right.
\end{equation}


\subsection{Construction of the resonant trajectory}


We now produce an example of perturbation $\beta(s)$ for which we can construct a resonant solution to \eqref{modbeta}.

\begin{proposition}[resonant trajectory]
\label{laprop}
Let  $\beta(s)$ be defined as the function 
\begin{equation}
\label{beta}
\beta(s) = - \frac{\sin(4s)}{ s \log (s)}\quad \mbox{for} \quad s > 0. 
\end{equation}
There exists $s_0 >0$ and  $(a_0,\theta_0)$ and for all $k \in \N$, constant $B_k$ such that the solution of \eqref{modbeta} with initial data $(a(s_0),\theta(s_0)) =  (a_0,\theta_0)$ exists for all $s \in [s_0,+\infty)$ and satisfies $a(s) \geq 2$ and 
\begin{equation}
\label{hash}
\left|
\begin{array}{rcl}
a(s) &=& \frac{1}{4} \log s + c(s)\\[2ex]
\theta(s) &=& 4 s + \psi(s)
\end{array}
\right.
\quad\mbox{with} \quad \forall\, k \in \N, \quad 
\Big|\frac{\dd^k c }{\dd s^k} (s)\Big| \leq B_k \frac{\log s}{s} \quad \mbox{and} \quad  \Big|\frac{\dd^k\psi}{\dd s^k} (s)\Big|  \leq \frac{B_k}{s}. 
\end{equation}
\end{proposition}

\begin{proof}[Proof of Proposition \ref{laprop}]  We use the classical method of backwards in time integration of the flow to construct the solution with the suitable behaviour at $+\infty$.\\

\noindent{\bf step 1} Change of variables.  
Let us set $2a(s) = \cosh(r(s)) \geq 1$. 
As long as $r(s) > 0$, we have $\sqrt{4a^2 - 1} = \sinh(r)$ and the system \eqref{modbeta} can be written
$$
\left|
\begin{array}{rcl}
\theta_s &=& \displaystyle 4 + 2 \beta(s) -  2\beta(s) \frac{\cosh(r)}{\sinh(r)} \cos(\theta)\\[2ex]
r_s &=& - 2\beta(s)\sin(\theta). 
\end{array}
\right.
$$
Let $\psi = \theta - 4s$, we have 
\begin{equation}
\label{plok}
\left|
\begin{array}{rcl}
\psi_s &=&  2 \beta(s) - 2 \beta(s) \frac{1 + e^{-2r}}{1 - e^{-2r}} ( \cos(\psi ) \cos(4s) - \sin(\psi)\sin(4s)) \\[2ex]
r_s &=& - 2\beta(s)( \sin(\psi) \cos(4s) + \cos(\psi) \sin(4s) ). 
\end{array}
\right.
\end{equation}
Setting $\rho(s) = r(s) - \log\log s$, we have 
\begin{equation}
\label{syslap}
\left|
\begin{array}{rcl}
\psi_s &=&  2 \beta(s) - 2 \beta(s) (1 + f(s,\rho))  ( \cos(\psi ) \cos(4s) - \sin(\psi)\sin(4s)) \\[2ex]
\rho_s &=& \displaystyle - \frac{1}{s \log s} - 2\beta(s)( \sin(\psi) \cos(4s) + \cos(\psi) \sin(4s) ). 
\end{array}
\right.
\end{equation}
with 
$$
f(s,\rho) =  \frac{1 + e^{-2r}}{1 - e^{-2r}} - 1 = \frac{2 e^{-2r}}{1 - e^{-2r}} = \frac{2 e^{-2\rho}}{(\log s)^2 - e^{-2\rho}}. 
$$

\noindent{\bf step 2} Backward bounds. We now derive uniform backward bounds which are the heart of the argument.

\begin{lemma}[Uniform backward bounds]
\label{cvneoineoneonvi}
For all $M > s_0$,  let us define $(\rho^M(s), \psi^M(s))$ be the solution of the system \eqref{syslap} such that 
$(\rho^M(M), \psi^M(M)) = (0,0)$. 
There exists a constant $B$ and $s_0$ sufficiently large such that for all $M > s_0$,  $(\rho^M, \psi^M)$ exists on $[s_0,M]$, and moreover, 
\begin{equation}
\label{samedi}
\forall\, M > s_0, \quad 
\forall\, s \in [s_0,M] \qquad
|\rho^M(s) |\leq \frac{B}{s} \quad\mbox{and}\quad |\psi^M(s)| \leq \frac{B}{s}. 
\end{equation}
Moreover, for all $k$ and $n$ in $\N$ there exists a constant $B_{k,n}$ such that for all $M$ and all $s \in [s_0,M]$
\begin{equation}
\label{lolo}
\left|\frac{\partial^{k+n} f}{\partial s^k \partial \rho^n}  (s, \rho^M(s))\right| \leq \frac{B_{k,n}}{s^k}
\end{equation}
\end{lemma}
\begin{proof}[Proof of Lemma \ref{cvneoineoneonvi}]
Note first that if $|\rho(s)| \leq \frac{B}{s}$ for $s \geq s_0$, we have $(\log s)^2 - e^{-2 \rho(s)} > (\log s_0)^2 - e^{\frac{2B}{s_0}} \geq 1$ if for instance 
\begin{equation}
\label{cond1}
s_0 \geq 2B \geq \exp(\sqrt{ e + 1}). 
\end{equation} 
Hence under these conditions, $f(s,\rho) \in [0,2 e]$ and all its derivative with respect to $\rho$ and $s$ satisfy bounds of the form \eqref{lolo}. We deduce that under the condition \eqref{cond1},  when $|\rho(s) |\leq \frac{B}{s}$ and $|\psi(s)| \leq \frac{B}{s}$, then we have 
\begin{equation}
\label{estderiv}
|\rho_s(s)| + |\psi_s(s)| \leq \frac{30}{s \log s}. 
\end{equation}
For all $M > s_0$, define 
$$
T_M(s_0,B) = \inf\{ s \in [s_0,M] \quad\mbox{s. t.}\quad  \forall \sigma \in [s,M]\, \quad  |\rho^M(\sigma) |\leq \frac{B}{\sigma} \quad\mbox{and}\quad |\psi^M(\sigma)| \leq \frac{B}{\sigma} \}. 
$$
As $\rho^M(M) = \psi^M(M) =  0$, the previous estimate show that under the condition \eqref{cond1} the flow exists locally for such initial condition, and we have $T_M(s_0,B) > 0$. 
We will show that there is a choice of $B$ and $s_0$ such that for all $M$, $T_M(s_0,B) = s_0$. 
\\
Let $M > 0$, assume that $s$ is such that for all $\sigma$, $\rho^M(\sigma)$ and $\psi^M(\sigma)$ satisfy the bound \eqref{samedi} for $\sigma \in [s,M]$. We thus have as $\psi^M(M) = 0$
\begin{multline*}
\psi^M(s) = \int_M^{s} 2 \beta(\sigma) \dd \sigma  \\
-\int_M^s 2 \beta(\sigma) (1 + f(\sigma,\rho^M))   \cos(\psi^M ) \cos(4\sigma)\dd \sigma 
+ \int_M^s 2 \beta(\sigma) (1 + f(\sigma,\rho^M))   \sin(\psi^M)\sin(4\sigma)) \dd \sigma.
\end{multline*}
Let us calculate the three contributions to the right-hand side: 
$$
\int_M^{s} 2 \beta(\sigma) \dd \sigma = - \int_M^{s} \frac{2 \sin(4\sigma)}{\sigma \log \sigma}\dd \sigma = \left[ \frac{ \cos(4\sigma)}{ 2 \sigma \log \sigma}\right]_{M}^s  +  \int_M^{s} \frac{\cos(4\sigma)( \log \sigma + 1)}{2 \sigma^2 (\log \sigma)^2}\dd \sigma. 
$$
Thus, there exists $B_0$ independent of $M$ such that for all $s < M$ this term is bounded in absolute value by $\frac{ B_0}{s}$. 
The second term can be written 
\begin{multline*}
\int_M^s  \frac{(1 + f(\sigma,\rho^M))}{\sigma\log \sigma}   \cos(\psi^M ) \sin(8\sigma)\dd \sigma  = - \left[ \frac{(1 + f(\sigma,\rho^M))}{ 8 \sigma\log \sigma}   \cos(\psi^M ) \cos(8\sigma) \right]_M^s \\
+ \int_M^s  \cos(8\sigma) \frac{\dd}{\dd s}\left(\frac{(1 + f(\sigma,\rho^M))}{ 8 \sigma\log \sigma}   \cos(\psi^M )\right) \dd \sigma. 
\end{multline*}
Using \eqref{estderiv} and \eqref{lolo}, we see that this term 
can be bounded by $\frac{B_0}{s}$ up to a increasing of the constant $B_0$. 
The last term can be bounded by 
$$
2\left| \int_M^s \frac{\sin^2(4\sigma) }{\sigma \log \sigma } (1 + f(\sigma,\rho^M))   \sin(\psi^M) \dd \sigma \right| \leq 6 \int_s^M \frac{|\psi^M(\sigma)|}{\sigma \log \sigma} \dd \sigma \leq \frac{6 B}{\log s_0} \int_s^M \frac{1}{\sigma^2} \dd \sigma. 
$$
So far we have proved that for all $s < M$, 
\begin{equation}
\label{estpsi}
|\psi^M(s)| \leq  \frac{2 B_0}{s} +  \frac{6B }{s \log s_0 } \leq \frac{B}{2s}, 
\end{equation}
provided we take $B > 8 B_0$ and $\log s_0 > 24$. 
Now we have 
$$
\rho^M(s) - \rho^M(M)  = \int_M^s \left(  
- \frac{1}{\sigma \log \sigma} + \frac{2\sin(4\sigma)}{\sigma \log \sigma}( \sin(\psi^M) \cos(4\sigma) + \cos(\psi^M) \sin(4\sigma) )\right) \dd \sigma.
$$
Using $2\sin^2(4\sigma) = 1 - \cos(8\sigma)$ and $2\sin(4 \sigma) \cos(4 \sigma) = \sin(8 \sigma)$ 
and the fact that $\rho^M(M) = 0$, we obtain 
$$
\rho^M(s) = 
- \int_{s}^N \frac{\sin(\psi^M) \sin(8\sigma) }{\sigma \log(\sigma)} \dd \sigma - \int_{s}^N \frac{\cos(\psi^M) - 1 }{\sigma \log(\sigma)}  \dd \sigma
+ \int_{s}^N \frac{\cos(\psi^M) \cos(8\sigma) }{\sigma \log(\sigma)}  \dd \sigma. 
$$
The first two terms can be treated by integration by part as before, and we can show that they can be bounded by $\frac{B_0}{s}$ after a possible increase of $B_0$ which is a constant independent of $M$, $s$ and $B$. Then 
using $|\cos(\psi) - 1 |\leq |\psi|^2 \leq \frac{B^2}{\sigma^2}$ we thus see that we have 
\begin{equation}
\label{estrho}
|\rho^M(s)| \leq \frac{2 B_0}{s} + \frac{B^2}{s (s_0 \log s_0)} \leq \frac{B}{2s}. 
\end{equation}
provided that $B > 8 B_0$ and $s_0 \log s_0 > 4B$.
Hence if $B$ and $s_0$ are large enough to satisfy condition \eqref{cond1} and the other conditions above, then \eqref{estpsi} and \eqref{estrho} are satisfy for all $s \geq T_M(s_0,M)$ which shows that $T_M(s_0,B) = s_0$. 
\\
The last estimate is then easily proved. 
\end{proof}

\noindent{\bf step 3} Conclusion. Let us take $N$ and $M$ such that $s_0 < N \leq M$. Using \eqref{lolo}, we have that 
$$
|\rho^M_s - \rho^N_s| + |\psi^M_s - \psi^N_s| \leq   \frac{C}{s \log s} ( |\rho^M(s) - \rho^N(s)| + |\psi^M(s)- \psi^N(s)|),
$$
for some constant $C$ independent of $M$ and $N$. 
Hence  for all $s \in [s_0,N]$, by integrating between $s$ and $N$, using the condition $\rho^N(N) = \psi^N(N) = 0$ and the bound \eqref{samedi}, we have 
\begin{multline*}
|\rho^M(s) - \rho^N(s)| +  |\psi^M(s) - \psi^N(s)| \leq \\
 \frac{2B}{N} + \int_{s}^N 
\frac{C}{\sigma \log \sigma} ( |\rho^M(\sigma) - \rho^N(\sigma)| + |\psi^M(\sigma)- \psi^N(\sigma)|)   \dd \sigma 
\end{multline*}
By using Gr\"onwall's lemma (see Lemma \ref{gr} in Appendix) we obtain 
\begin{eqnarray*}
|\rho^M(s) - \rho^N(s)| +  |\psi^M(s) - \psi^N(s)| 
&\leq& \frac{2B}{N} + \frac{2B}{N}\int_s^N \frac{C}{\sigma \log \sigma} \exp \left(\int_{s}^{\sigma} \frac{C}{\tau \log \tau} \dd \tau\right) \dd \sigma \\
&\leq& \frac{2B}{N} + \frac{2B}{N}\int_s^N \frac{C(\log \sigma)^{C-1}}{\sigma}  \dd \sigma \\
&\leq& \frac{2B}{N} + \frac{2BC}{N} (\log N)^{C}. 
\end{eqnarray*}
We deduce that the sequence of function $(\rho^M,\psi^M)_{M \in \N}$ is Cauchy and thus converges on every interval $[s_0,T]$ for any fixed $T > s_0$. This solution solves the system \eqref{syslap} on this interval and does not depend on $T$ as it coincides with the unique solution on \eqref{syslap} with initial value $(\rho(s_0), \psi(s_0)) = \lim_{M \to \infty} (\rho^M(s_0),\psi^M(s_0))$. Hence this solution exists globally and satisfies the bound \eqref{samedi}. 

Moreover, by using \eqref{syslap}, we easily see that for all $k \geq 1$, there exists $B_k$ such that 
$$
\forall\, s \in [s_0,+\infty), \quad \Big|\frac{\dd^k\rho }{\dd s^k} (s)\Big| + \Big|\frac{\dd^k\psi}{\dd s^k} (s)\Big| \leq \frac{B_k}{s \log s}. 
$$
We deduce that $\theta(s) = \psi(s) + 4s$ satisfy the hypothesis of the theorem.
Moreover, we have $2a(s) = \cosh(r(s)) = \cosh( \log \log s + \rho(s))$. Hence 
$$
4a(s) = 2\cosh(r(s)) = e^{ \log \log s + \rho(s)} + e^{-  \log \log t - \rho(s)} = (\log s)  e^{\rho(s)} + \frac{e^{-\rho(s)}}{\log s}
$$
from which we easily deduce the result with $4 c(s) = \log (s) ( e^{\rho(s)} - 1) + \frac{e^{-\rho(s)}}{\log s}$. 
Finally, we check that  as $|\rho(s)| \leq \frac{B}{s}$ we have 
$$
a(s) \geq \frac{1}{4}\log(s_0) e^{\frac{B}{s_0}} \geq 2,
$$
provided $s_0$ is large enough. 

\end{proof}


\subsection{Energy drift}


The solution constructed in Proposition \ref{laprop} exhibits a monotonic growth of the energy.

\begin{corollary}[Logarithmic growth of the energy]
\label{coro1}
With $\beta$ given by \eqref{beta}, there exists $s_0$, $L(s_0) >0$ and $b(s_0)$ and positive constants $(B_k)_{k\in \N}$ and $(\alpha_k)_{k \in \N}$ such that the system \eqref{perturb0} admits global solutions $L(s)$ and $b(s)$ on $[s_0,+\infty)$ such that 
\begin{equation}
\label{boundL2}
\forall\,s \in [s_0,+\infty), \quad   \frac{1}{B_0 \log s} \leq L^2  \leq B_0 \log s \quad \mbox{and} \quad |b(s) | \leq B_0 (\log s)^3,
\end{equation}
and 
$$
E(b,L) = \frac{1}{L^2}(\frac{b^2}{4} + 1) +   L^2 = 4 a = \log s + \mathcal{O}(\frac{\log s}{s}), \quad \mbox{when} \quad t \to +\infty,
$$
and such that we have the bounds for $k \geq 1$, 
\begin{equation}
\label{bounderivs}
\left|\frac{\dd^k L}{\dd s^k} (s)\right| + \Big|\frac{\dd^k b}{\dd s^k} (s)\Big| + \Big|\frac{\dd^k }{\dd s^k}\Big(\frac{1}{L}\Big) (s)\Big| \leq B_k (\log s)^{\alpha_k}. 
\end{equation}

Moreover the time $t(s)$ satisfying $\frac{\dd t}{\dd s} = L^2 > 0$ and $t(s_0) = s_0$ satisfies 
\begin{equation}
\label{boundtime}
| t(s) - s| \leq B_0 (\log s)^2, \quad \mbox{for} \quad s > s_0. 
\end{equation}
Hence as $t(s)$ is increasing, it is globally invertible. Moreover, by denoting $L(t)$ and $b(t)$ the quantities $L$ and $b$ viewed as depending on the time $t$, we have the bounds. 
\begin{equation}
\label{bounderivt}
\left|\frac{\dd^k L}{\dd t^k} (t)\right| + \Big|\frac{\dd^k b}{\dd t^k} (t)\Big| + \Big|\frac{\dd^k }{\dd t^k}\Big(\frac{1}{L}\Big) (t)\Big| \leq B_k (\log t)^{\alpha_k}, 
\end{equation}
for $t \geq t_0 := s_0$. 
\end{corollary}
\begin{proof}[Proof of Corollary \ref{coro1}]
From \eqref{changeco} we have the explicit formulae 
$$
\frac{1}{L^2} = 2a  - \sqrt{4a^2 - 1}\cos(\theta)  \quad \mbox{and}\quad
b  = 4L^2  \sqrt{4a^2 - 1} \sin(\theta) 
$$
from which we easily deduce the existence of $L(s) >0$ and $b(s)$ as $a(s) > 2$. We have 
$$
\frac{1}{4a } \leq 2a - \sqrt{4a^2 - 1} \leq \frac{1}{L^2} \leq 2a + \sqrt{4a^2 - 1} \leq 4a 
\quad \mbox{and}\quad |b| \leq 8 a L^2  \leq 32 a^3
$$
and \eqref{boundL2} can be obtained using \eqref{hash}. Moreover, we can assume that $s_0$ is large enough to ensure that $a(s) \geq 1 + \frac{1}{8}\log(s)$ and hence $\sqrt{4a^2 - 1} \geq \frac{1}{4}\log s$ for $s \geq s_0$. 
Using \eqref{hash}, we have that  $\frac{\dd^k }{\dd s^k}a(s) = \mathcal{O}( \frac{\log s}{s})$ and $\frac{\dd^k}{\dd s^k}\theta(s)(s) = \mathcal{O}(1)$ for $k \geq 1$. Using Fa\`a di Bruno formula, we deduce that 
for some constants $B_k$, we have for all $k \geq 0$ and all $s \geq s_0$, 
$$
\left|\frac{\dd^k }{\dd s^k}  \sqrt{4a^2 - 1}\right|\leq B_k \log s  \quad\mbox{and} \quad \left|\frac{\dd^k }{\dd s^k} \cos (\theta)\right | + \left|\frac{\dd^k }{\dd s^k} \sin (\theta)\right | \leq B_k
$$
and then   \eqref{bounderivs} by using again Fa\`a di Bruno formula and the bound $L^{\pm 2} \leq C (\log s)$ for some constant $C$ depending only on $s_0$. We can check that $\alpha_k = \mathcal{O}(k)$. 
%
%
%

From Equation \eqref{fish}, we have
$$
\frac{\dd t}{ \dd s} = L^2 (\theta,a) = \left( 1 + 2 \sum_{n > 0}  \left( \frac{2a - 1}{2a + 1} \right)^{\frac{n}{2}} \cos(n \theta) \right).
$$
Assuming $t(s_0) = s_0$, and using $(2a - 1) < (2 a + 1)$ which justifies the infinite sums in $n$, 
\begin{eqnarray*}
t(s) - s &=& 2 \sum_{n > 0}   \int_{s_0} ^s \left( \frac{2a - 1}{2a + 1} \right)^{\frac{n}{2}} \cos(4 \sigma n + n \psi) \dd \sigma \\
&=& 2 \sum_{n > 0}   \int_{s_0}^s \left( \frac{2a - 1}{2a + 1} \right)^{\frac{n}{2}} \cos(4 \sigma n)\cos(n \psi) \dd \sigma\\
&& - 2 \sum_{n > 0}   \int_{s_0}^s \left( \frac{2a - 1}{2a + 1} \right)^{\frac{n}{2}} \sin(4 \sigma n) \sin(n \psi) \dd \sigma. 
\end{eqnarray*}
We now integrate by part the terms in the right-and side. For the first term we use
\begin{multline*}
\int_{s_0}^s \left( \frac{2a - 1}{2a + 1} \right)^{\frac{n}{2}} \cos(4 \sigma n)\cos(n \psi) \dd \sigma 
= \\
\frac{1}{4n}\left[ \left( \frac{2a - 1}{2a + 1} \right)^{\frac{n}{2}} \sin(4 \sigma n)\cos(n \psi)  \right]_{s_0}^s
+  \frac{1}{4n} \int_{s_0}^s \psi_s  \left( \frac{2a - 1}{2a + 1} \right)^{\frac{n}{2}} \sin(4 \sigma n)\sin(n \psi) \dd \sigma
\\
-\int_{s_0}^s  \frac{2a_s}{(2 a + 1)^2} \left( \frac{2a - 1}{2a + 1} \right)^{\frac{n}{2} - 1} \sin(4 \sigma n)\cos(n \psi) \dd \sigma
\end{multline*}
As $\psi_s = \mathcal{O}(\frac{1}{s})$ (see \eqref{hash}),  
the global contribution of the first term is bounded by
\begin{multline*}
C \int_{s_0}^{s}  \sum_{n \geq 1} \frac{1}{\sigma }  \left( \frac{2a - 1}{2a + 1} \right)^{\frac{n}{2}} \dd \sigma 
\leq C \int_{s_0}^{s} \frac{1}{\sigma}\frac{1}{2a - \sqrt{4a^2 - 1}}  \dd \sigma \\
\leq C  \int_{s_0}^{s} \frac{a(\sigma)}{\sigma}\dd \sigma 
\leq C \int_{s_0}^s \frac{\log(\sigma)}{\sigma} \dd \sigma \leq C (\log s)^2. 
%
%
%
\end{multline*}
by using the formula for the Poisson Kernel, and up to modifications of the constant $C$ in each inequalities, depending only on $s_0$ and numerical constants. 
As $a_s = \mathcal{O}(\frac{\log (s) }{s})$ and $\log \sigma / (2 a(\sigma) + 1)$ is bounded, we calculate that    the third term yields similarly a contribution of order $\mathcal{O}((\log s)^2)$. 
We then deduce that 
$$
| t(s) -s | \leq  B_0 (\log s)^2,
$$
for some constant $B_0$. The bounds 
\eqref{bounderivt} then easily derive from the bounds \eqref{bounderivs}, \eqref{boundL2} and the $\frac{\dd}{\dd s} = L^2 \frac{\dd}{\dd t}$. 
\end{proof}


\section{Main result for the harmonic oscillator}


We are now in position to construct the unbounded trajectory of Theorem \ref{Th1}. The construction relies on the existence of the resonant finite dimensional trajectory of Proposition \ref{laprop}, and the backwards integration method for the full PDE as introduced in \cite{Me90}.


\subsection{Construction of the resonant trajectory}


We consider the equation \eqref{LHO1}. From classical grounds, the a priori bound for all $t$ and all $k \in \N$, 
$$
\Norm{\partial_t^{k} V(t,x)}{\Hc^s} < +\infty. 
$$
ensures the existence and uniqueness of global solutions to \eqref{LHO1} in $\Hc^r$ for $r > 1$ satisfying $u(t_0,x) = u_0(x) \in \Hc^r$, with a given $t_0 \in \R$.They are solutions of the equation 
$$
\forall\, t \in \R \quad u(t,x) = e^{- i (t - t_0) H }u_0(x) - i \int_{t_0}^t e^{- i (t - t_0 - s)H } V(s,x) u(s,x) \dd s. 
$$

We make the change of unknown \eqref{prince} with the modulation system \eqref{perturb0} associated with the function $\beta(s)$ defined in \eqref{beta}. By using \eqref{eqH2} we obtain the system  
\begin{equation}
\label{mend}
i \partial_s v = H v + \gamma_s(s) v +   \beta(s) |y|^2 v + W(s,y) v, \quad y \in \R^d. 
\end{equation}
where 
$$
W(s,y) = L^2 V(t,x)
$$
with $s$ and $y$ satisfying \eqref{prince}.  The heart of the proof of Theorem \ref{Th1} is the following statement.

\begin{proposition}[resonant trajectory in renormalized variables]
\label{prop1}
Let $\beta(s)$ be given by the formula \eqref{beta} and 
\begin{equation}
\label{eqW}
W(s,y) = -  \alpha \beta(s) h_0(y) \quad \mbox{with} \quad  \alpha = \frac{ ( h_1 , |y|^2 h_0)_{L^2}}{( h_1 , h_0 h_0)_{L^2}} = - \frac{9 \sqrt{\pi}}{2}. 
\end{equation}
Let $s_0$ be as in  Proposition \ref{laprop} and $r > 1$. There exists a constant $B$ and a solution $\gamma(s) = - \lambda_0 s = - 2s$ and  $v(s) \in \Hc^r_{\rad}$ to the equation \eqref{mend} such that 
$$
v(s,y) = h_0(y) + w(s,y),\quad \mbox{with}\quad \forall\, s \in (s_0,+\infty),\quad \Norm{w(s)}{\Hc^r} \leq \frac{B}{\sqrt{s}}. 
$$
\end{proposition}
\begin{proof}[Proof of Proposition \ref{prop1}]
Let us take $\gamma_s = - \lambda_0 = - 2$. 
Let $ w = v - h_0$. As $H h_0 = \lambda_0 h_0$, we have that 
\begin{equation}
\label{cador}
i\partial_s w = (H - \lambda_0) w +  \beta(s) |y|^2 w + W(s,y) w + R(s)
\end{equation}
with 
$$
R(s) = \beta(s) |y|^2 h_0 + W(s,y) h_0 = \beta(s) ( |y|^2 h_0 - \alpha h_0^2). 
$$
By definition of $W$, we have $(R(s),h_1)_{L^2} = 0$. 
Let $g = e^{i s (H - \lambda_0)} w$. We have 
\begin{equation}
\label{eqg}
i \partial_s g =  K(s) g + e^{i s (H - \lambda_0)} R(s). 
\end{equation}
where 
\begin{eqnarray}
K(s) &=& \beta(s) e^{i s H}|y|^2 e^{- is H} +  e^{i s H} W(s,y) e^{- i s H} \nonumber \\[2ex]
&=& \beta(s) e^{i s H}(|y|^2 - \alpha h_0) e^{- is H}
\label{Ks}
\end{eqnarray}
Note that by using \eqref{eqflow}, \eqref{algebra} and \eqref{y2op} we have that for all $r > 1$ and all $f \in \Hc^{r+1}_{\rad}$, 
\begin{equation}
\label{estKs}
\Norm{K(s) f}{\Hc^r} \leq \frac{C_r}{s \log s} \Norm{f}{\Hc^{r+1}}
\end{equation}
for some constant $C_r$ independenf of $f$. Moreover, for all $r >1$, we have 
$$
[H^{s/2},K(s)] = \beta(s) e^{i s H} ( [H^{r/2},|y|^2]  - \alpha [H^{r/2},h_0 ]) e^{-i sH}
$$
by hence using \eqref{eqcom}, \eqref{algebra} and the fact that $[H^{r/2},h_0]f = H^{r/2}(h_0 f)  - h_0 ( H^{r/2} f)$, we have 
\begin{equation}
\label{estcommKs}
\Norm{[H^{r/2},K(s)] f}{L^2} \leq \frac{C_r}{s \log s} \Norm{f}{\Hc^{r}}, 
\end{equation}
for $r > 1$ and a constant $C_r$ depending only on $r$. 
Note that for all $s$, $K$ is Hermitian, {\em i.e.} for all $\forall\, f,g \in \Hc^{r+2}$, 
$$
(f, K(s) g)_{L^2} = (K(s) f , g)_{L^2}. 
$$

For $M \in (s_0,+\infty)$, let $w^M$ be the solution of \eqref{cador} in $\Hc_{\rad}^s$ such that $w^M(M) = 0$, $g^M = e^{i s (H - \lambda_0)} w^M$, and 
$$
f^M = g^M - i \int_s^M e^{i \sigma (H - \lambda_0)} R(\sigma,y) \dd \sigma  =: g^M - r^M. 
$$
We have that $f^M(M) = 0$, and 
\begin{equation}
\label{eqfM}
i \partial_s f^M = K(s) g^M =  K(s) f^M 
+ R^M(s) 
\end{equation}
with
\begin{equation}
\label{eqRM}
R^M(s) =  K(s) r^M. 
\end{equation} 
Now we have 
$$
r^M = i \int_s^M e^{i \sigma (H - \lambda_0)} R(\sigma,y) \dd \sigma =  i \int_s^M \beta(\sigma) e^{i \sigma (H - \lambda_0)} (|y|^2 h_0 -  \alpha h_0^2)  \dd \sigma
$$
Let us calculate the terms in the right-hand side. We have 
$$
e^{i \sigma (H - \lambda_0)}(|y|^2 h_0 -  \alpha h_0^2)  = \sum_{n \neq 1} e^{i 4 n \sigma } ( ( h_n, |y|^2 h_0^2)_{L^2} - \alpha (h_n,h_0^2)_{L^2})  h_n
$$
as the term for $n = 1$ vanishes by definition of $\alpha$. 
This shows that 
$$
r^M(s) = - \sum_{n \neq 1}  \left( \int_s^M \frac{\sin(4\sigma)e^{4i n \sigma}}{\sigma \log \sigma} \dd \sigma\right) ( ( h_n, |y|^2 h_0^2)_{L^2} - \alpha (h_n,h_0^2)_{L^2})  h_n. 
$$
But as $n \neq  1$, we have after integration by part, that for some constant $B$ independent on $M$, $n$ and $s$, we have for $s \leq M$,
$$ 
\left| \int_s^M \frac{\sin(4\sigma)e^{4i n \sigma}}{\sigma \log \sigma} \dd \sigma\right| \leq \frac{B}{s}. 
$$ 
This shows that 
\begin{multline}
\label{raslbol}
\Norm{r^M(s)}{\Hc^r}^2 \leq \frac{2 B^2}{s^2} \left( \sum_{n \in \N} \langle n \rangle^r ( h_n, |y|^2 h_0^2)_{L^2}^2 
+ \alpha^2 \sum_{n\in \N} \langle n \rangle^r (h_n,h_0^2)_{L^2}^2\right) \\ \leq  \frac{2 B^2}{s^2} (\Norm{|y|^2 h_0}{\Hc^r}^2 + \alpha^2 \Norm{h_0^2}{\Hc^r}) \leq \frac{C^2_r}{s^2}
\end{multline}
for some constant $C_r$ independent on $M$ and $s$ but depending on $r> 1$.
In view of the expression \eqref{eqRM} and the estimate \eqref{estKs}, we have 
$$
\Norm{R^M(s)}{\Hc^r} \leq \frac{C_{r}}{s \log s} \Norm{r^M(s)}{\Hc^{r +1}} \leq \frac{C_r}{s^2 \log s},  
$$ 
up to a modification of the constant $C_r$ in the last inequality. 
Now using the equation \eqref{eqfM} on $f^M$, we have 
\begin{eqnarray*}
 \partial_s \Norm{f^M}{\Hc^r}^2 &=&   \mathrm{Im} \left(    H^{r/2} f^M,  H^{r/2} K(s)  f^M )_{L^2} + (H^{r/2}f^M, H^{r/2} R^M)_{L^2}\right) \\
 &=& \mathrm{Im} \left(  (H^{r/2} f^M,  [H^{r/2}, K(s)] f^M )_{L^2} + (H^{r/2}f^M, H^{r/2} R^M)_{L^2}\right) 
\end{eqnarray*}
as the multiplication by $K$ defines a symmetric operator. Hence we have by using \eqref{estcommKs},  
$$
 \partial_s \Norm{f^M}{\Hc^r}^2 \leq  \frac{C}{s \log s}\Norm{f^M}{\Hc^r}^2 + 
 \frac{C}{s^2 \log s } \Norm{f^M}{\Hc^{r}},
$$
for some constant $C$ depending only on $r$. 
For $\varepsilon > 0$,
let $y^M_{\varepsilon}(s) = \sqrt{\Norm{f^M}{\Hc^r}^2 + \varepsilon^2}$. The previous inequality implies that for all $\varepsilon$, we have 
$$
\partial_{s} (y^M_{\varepsilon})^2  = 2 y^M_{\varepsilon} \partial_s y^M_{\varepsilon} \leq  
\frac{C}{s \log s}(y^M_{\varepsilon})^2 + 
 \frac{C}{s^2 \log s } y^M_{\varepsilon}, 
$$
and hence as $y^M_{\varepsilon} > 0$ for all $s$, 
$$
 \partial_s y^M_{\varepsilon} \leq  
\frac{C}{2 s \log s}y^M_{\varepsilon} + 
 \frac{C}{2 s^2 \log s }.
$$
By using Gr\"onwall's lemma (see Lemma \ref{gr} below), we obtain as $y^M_{\varepsilon}(M) = \varepsilon$, 
\begin{eqnarray}
y^M_{\varepsilon}(s) &\leq& \varepsilon +  \frac{C}{s} + \int_s^M \Big(\varepsilon +  \frac{C}{\sigma}\Big) \frac{C}{2 \sigma \log \sigma} \exp \left( \int_s^\sigma \frac{C}{2 \tau \log \tau} \dd \tau\right) \dd \sigma
\nonumber \\
&\leq & \varepsilon +  \frac{C}{s} + \int_s^M \Big(\varepsilon +  \frac{C}{\sigma}\Big) \frac{C (\log \sigma)^{C - 1}}{2 \sigma}  \dd \sigma. 
\label{top}
\end{eqnarray}
By letting $\varepsilon \to 0$, we deduce that for all $\alpha >0$, there exists a constant $\kappa_\alpha$ such that for all $M$ and all $s \in (s_0,M)$, 
\begin{equation}
\label{boundfM}
\Norm{f^M(s)}{\Hc^r} \leq \frac{\kappa_\alpha}{s^{1 - \alpha}}. 
\end{equation}
Now if we take $f^M$ and $f^N$ for $M > N$, we have 
$$
i \partial_s (f^M - f^N) =  K(s) (f^M - f^N)  + K(s) (r^M - r^N). 
$$
But we have 
$$
r^M - r^N = i \int_N^M e^{i \sigma (H - \lambda_0)} R(\sigma,y) \dd \sigma  = r^M(N). 
$$
In particular, we have using \eqref{estKs} 
$$
\int_s^N \Norm{K(\sigma) (r^M - r^N)}{\Hc^r} \dd \sigma \leq C_r \int_{s}^{N} \frac{1}{\sigma \log \sigma} \Norm{r^M(N)}{\Hc^{r+1}}\dd \sigma.  
$$
Hence we have 
\begin{multline*}
\Norm{ f^M(s) - f^N(s)}{\Hc^r} \leq \Norm{f^M(N)}{\Hc^r} +  C_r \log \log N \int_s^N \Norm{r^M(N)}{\Hc^{r+1}} \\
+  \frac{C}{\sigma \log \sigma}  \Norm{f^M(\sigma) - f^N(\sigma)}{\Hc^r}\dd \sigma
\end{multline*}
and by Gr\"onwall estimate, \eqref{raslbol} and \eqref{boundfM} with $\alpha = \frac{1}{4}$, 
$$
\Norm{ f^M(s) - f^N(s)}{\Hc^r}  \leq \frac{\kappa_{1/4}}{N^{3/4}}\left( 1 +   \int_s^N \frac{( \log \sigma)^{C-1}}{\sigma } \dd \sigma \right) \leq \frac{C}{\sqrt{N}}, 
$$
for some constant $C$ independent of $N$ large enough.
Hence the sequence of function $(f^M(s))_{M \in \N}$ is Cauchy and converge uniformly in $\mathcal{C}((0,T),\Hc^r))$ for all $T$. 

Moreover, the functions $r^M$ also converge to a function $r(s)$ on $(s_0,+\infty)$ in $\Hc^s$ and satisfies $\Norm{r(s)}{\Hc^r} \leq C_r/s$ (see \eqref{raslbol}). We deduce that $g^M = f^M - r^M$ converges towards the unique solution of \eqref{eqg}. Moreover, by using \eqref{boundfM} with $\alpha = \frac{1}{2}$, 
we have $\Norm{g(s)}{\Hc^r} \leq \frac{B}{\sqrt{s}}$ for all $s \in (s_0,+\infty)$ and some constant $B$ depending on $r$. We obtain the result by noticing that $v = h_0 + e^{- i s (H - \lambda_0) } g$. 
\end{proof}


\subsection{Proof of Theorem \ref{Th1}}


It follows directly from the following quantitative version.

\begin{proposition}[Existence of the resonant trajectory]
\label{vnbeibvebveiuve}
Let $s_0$, $(b(t),L(t))$ and $s(t)$ satisfying Corollary \ref{coro1} and let $V(t,x)$ be the function defined as the time dependent Gaussian
\begin{equation}
\label{vbeivbibeveievb}
V(t,x) = -  \frac{9 \sqrt{\pi}}{2 L(t)^2}  \frac{\sin(4s(t))}{ s(t) \log s(t) }h_0(\frac{x}{L(t)}). 
\end{equation}
Then we have for all $k$ and all $r$,  
$$
\lim_{t \to \infty} \Norm{\partial_t^{k} V(t,x)}{\Hc^r} = 0. 
$$
Moreover, if $v(s,y)$ denote the function constructed in Proposition \ref{prop1}, then 
$$
u(t,x) = \frac{1}{L(t)} e^{- 2is(t)  - i \frac{b}{4}L^{-2}(t) |x|^2}  v(s(t), \frac{x}{L(t)})
$$
is a solution in $\Hc^r_\rad$, $r > 1$ of the equation 
$$
i \partial_{t} u = - \Delta u + |x|^2 u + V(t,x) u  
$$
on $(s_0,+\infty)$. Moreover, we have 
\begin{equation}
\label{decomp}
u(t,x) = u_0(t,x) + u_1(t,x)
\end{equation}
with 
$$
\Norm{u_0(t,x)}{\Hc^1}^2 \sim \log t \quad \quad\mbox{when}\quad  t \to \infty
$$
and such that for all $r > 1$, there exists $C_r$ and $\alpha_r$ such that 
$$
\quad \Norm{u_1(t,x)}{\Hc^r}  \leq C_r \frac{(\log t)^{\alpha_r}}{\sqrt{t}}. 
$$
In particular, we have 
\begin{equation}
\label{growth1}
\Norm{u(t,x)}{\Hc^1}^2 \sim \log t \quad \mbox{when}\quad  t \to \infty. 
\end{equation}
\end{proposition}
\begin{proof}[Proof of Proposition \ref{vnbeibvebveiuve}]
The bound on the potential are consequences of the estimates in Corollary \ref{coro1}. To prove \eqref{growth1}, we observe that $v(s,y) = h_0(y) + w(s,y)$ with $\Norm{w(s)}{\Hc^1} \leq \frac{B}{\sqrt{s}}$. Hence by using Lemma \ref{lemnorms}, we see that the contribution of $w(s,y)$ converges to $0$ in $\Hc^1$ norm, when $s$ goes to $\infty$. The bound on $u_1$
 are then easily proved by using the relation \eqref{boundtime} between $t(s)$ and $s$, and using the bound on $v$ and on $L$ and $b$. By using \eqref{normegaussienne}, we finally have 
$$
\Norm{u(t)}{\Hc^1}^2 = E(b(t),L(t)) + o(1) = 4a(t) + o(1) = \log(s(t)) + o(1) \sim \log(t). 
$$
and Proposition \ref{vnbeibvebveiuve} and Theorem \ref{Th1} are proved. 
\end{proof}

\section{The linear CR equation}
\label{sectioncr}

We consider in this section the linearized CR equation and propose a different approach to produce growth and realize Theorem \ref{Th1} using some specific properties of the CR equation. 


\subsection{Existence of resonant trajectories}


For CR, the existence of resonant trajectories for the perturbed problem can be reduced to the existence of suitable trajectories to the unperturbed flow.

\begin{proposition}[resonant trajectory near suitable trajectories]
\label{popop}
Assume that there exists $s_0$, $F(s,x)$ and $f(s,x)$ such that  
$$
i \partial_s f = \Tc[F] f, \quad s \in [s_0,+\infty], 
$$
and such that for all $r$ and $k$, there exists $\kappa = \kappa(r,k)$ and $C = C(r,k)$ such that 
\begin{equation}
\label{Ft}
\Norm{\partial_s^k F(s,x)}{\Hc^r} + \Norm{\partial_s^k f(s,x)}{\Hc^r} \leq  C e^{\kappa s}, \quad 
\end{equation}
and there exists $c$ and $\alpha$ such that 
\begin{equation}
\label{ft}
\Norm{f(s,x)}{\Hc^1} \sim c e^{\alpha s}, \quad s \to +\infty, \quad c,\alpha > 0. 
\end{equation}
Then there exists $V(t,x)$ and $u(t,x) = e^{-i t H}f(\log \log t ,x) + \mathcal{O}(1) $ realizing Theorem \ref{Th1}. 
\end{proposition}
\begin{proof}[Proof of Proposition \ref{popop}]
We set 
$$
V (t,x) = \frac{1}{t \log t}| e^{-i t H} F(\log \log t,x) |^2. 
$$
By using the exponential bounds for $F$, we easily verify that 
$V$ satisfies the decay hypothesis in time \eqref{decay} in $\Hc^r$ for $r > 1$ (to ensure the algebra property of $\Hc^r$). 
Let us set 
$$
u= e^{-i t H} v(t ,x). 
$$
Then $u$ solves \eqref{LHO1} if $v$ is solution of 
$$
i \partial_t v =\frac{1}{t \log t} e^{i t H }| e^{-i t H} F(s,x) |^2  e^{ - i tH } v, \quad s = \log \log t. 
$$
Let us decompose $F (s, x) = \sum_{k \in \Z} f_k(s) h_k(x)$ and $v (t, x) = \sum_{k \in \Z} v_k(s) h_k(x)$. The previous equation is equivalent to the collection of equations 
\begin{eqnarray*}
\forall\, k \in \N, \quad 
i \partial_t v_k(t)  &=&   \frac{1}{t \log t} \sum_{m,n,p \in \N} \chi_{k m n p} e^{4 i t (k - m + n  - p) }  F_{m}(s) \overline{F}_n(s) v_p(t) \\
&=& \frac{1}{t \log t} (\Tc[F] v)_k + \frac{1}{t \log t} (R(s,t)  v)_k, 
\end{eqnarray*}
where the coefficients  $\chi_{k m n p}$ are given by the formula \eqref{piaz} and 
$$
(R(s,\theta)  v)_k  = \frac{1}{t \log t} \sum_{ k \neq m - n + p } \chi_{k m n p} e^{4 i \theta (k - m + n  - p) }  F_{m}(s) \overline{F}_n(s) v_p, 
$$
define an operator $R(s,\theta)$ acting on $\Hc^r$ for $r > 1$ (see for instance \cite[Proposition 2.13]{GIP09}) which is oscillatory in $\theta$. 
We now set
$$
w(t,x) = v(t,x) - f(\log \log t, x). 
$$
Then by assumption on $f$,  $w$ satisfies 
$$
i \partial_t w = \frac{1}{t \log t} \Tc[F] w + \frac{1}{t \log t} R(s,t)  w  +  \frac{1}{t \log t} R(s,t)  f(s). 
$$
For $M$ large enough, we define $w^M(t)$ the solution  of this equation such that $w^M(M) = 0$, and we set 
$$
g^M(t,x) = w^M + i  \int_{t}^{M} \frac{1}{\sigma \log \sigma} R(\log \log \sigma,\sigma) f (\log \log \sigma) \dd \sigma = : w^M - r^M. 
$$
We have 
\begin{eqnarray*}
i \partial_t g^M &=& \frac{1}{t \log t} \Tc[F] g^M + \frac{1}{t \log t} R(s,t)  g^M + \frac{1}{t \log t} \Tc[F] r^M + \frac{1}{t \log t} R(s,t)  r^M \\
&=& K(t) g^M + K(t) r^M, \quad \mbox{with}\quad  K(t) = e^{i tH} V(t,x) e^{-i tH}, 
\end{eqnarray*}
a formula that can be compared with \eqref{eqfM}-\eqref{eqRM}. As $V$ is smooth, the operator $K(t)$ possess the same properties as the operator defined in the proof of Proposition \ref{prop21}, in particular the commutator estimate \eqref{estcommKs}. To conclude by using the same argumentation as in this proof, we thus just need to control $\Norm{r^M}{\Hc^r}$ for $r$ large enough (see estimate \eqref{raslbol}). 
Now we have 
$$
r_k^M(t) = i 
\sum_{ k \neq m - n + p } \chi_{k m n p}  \int_{t}^{M} \frac{1}{\sigma \log \sigma} e^{4 i \sigma  (k - m + n  - p) }  F_{m}(s(\sigma)) \overline{F}_n(s(\sigma)) f_p(s(\sigma))  \dd \sigma. 
$$
We integrate the oscillatory term by part and use the fact that $|k - m + n  - p|\geq 1$. Moreover, proposition 3.3 of \cite{GIP09} gives some bounds on the coefficients $\chi_{kmnp}$. By applying estimates for polynomials acting on $\Hc^r$, see
 \cite[Proposition 3.3]{GIP09}, we obtain a bound of the form (for $r > 1$)
\begin{eqnarray*}
\Norm{r^M(t)}{\Hc^r} &\leq& \frac{C}{t \log t} \Norm{F(\log \log t)}{\Hc^r}^{2}  \Norm{f(\log \log t)}{\Hc^r} \\
&&+ \frac{C}{M \log M} \Norm{F(\log \log M)}{\Hc^r}^{2}  \Norm{f(\log \log M)}{\Hc^r} \\
&&+ C \int_{t}^M\frac{1}{\sigma^2 (\log \sigma)^2 } \Norm{(\partial_s F)(\log \log \sigma)}{\Hc^r} \Norm{F(\log \log \sigma)}{\Hc^r}  \Norm{f(\log \log \sigma)}{\Hc^r} \dd \sigma\\
&&+ C \int_{t}^M\frac{1}{\sigma^2 (\log \sigma)^2 } \Norm{F(\log \log \sigma)}{\Hc^r}^2  \Norm{(\partial_s f)(\log \log \sigma)}{\Hc^r} \dd \sigma\\
&&+ C \int_{t}^M\frac{1}{\sigma^2 (\log \sigma)^2 } \Norm{F(\log \log \sigma)}{\Hc^r}^2  \Norm{f(\log \log \sigma)}{\Hc^r} \dd \sigma.
 \end{eqnarray*}
Using the bound on $f$ and $F$, we conclude that for some constants $C_r$ and $\beta_r$, we have for $t < M$ and uniformly in $M$, 
 $$
\Norm{r^M(t)}{\Hc^r} \leq C_r  \frac{(\log t)^{\beta_r}}{t}. 
 $$
 If we compare with \eqref{raslbol}, we see that we loose a factor $(\log t)^{\beta_r}$ compared with the estimates in the proof of Proposition \ref{prop21}, but it does not affect the result, and the conclusion is the same, see in particular \eqref{top} with same the kind of estimate. We conclude that $w^M$ converges towards a solution of \eqref{LHO1} such that in $\Hc^r$ for $r$ large enough, 
 $$
 u = e^{- i tH} (f(\log \log t) + \mathcal{O}(1))
 $$
 from which we obtain the result by using \eqref{ft}. 
\end{proof}


\subsection{Existence of suitable trajectories and conclusion}


Hence we are reduced to the problem of finding functions $F$ and $f$ satisfying \eqref{Ft} and \eqref{ft}. Due to the numerous invariance of the CR equation, there are many ways to construct such example. Up to a change of time one such example is given in \cite{Tho20} by using the analysis in \cite{ST20} for the lowest Landau level equation which coincide with the CR equation on the Bargmann-Fock space, see \cite{GHT16}. 
Here we give a general recipe to build simple examples based on the following fact: 

\begin{lemma}
Let $\kappa$, $\nu$, $\mu$ real numbers. There exists  $\beta \in \C$ and $\lambda \in \R$ such that 
\begin{equation}
\label{eqmode}
(\nu \Delta + i \mu ( 1+ \Lambda) + \kappa |x|^2 )h_0 +\Tc[ h_0 + \beta h_1] h_0 = \lambda h_0. 
\end{equation}
\end{lemma}
\begin{proof}
Using \eqref{eqy2}, \eqref{eqLap} and \eqref{bemol}, we have 
$$
\kappa |x|^2 h_0  =  \kappa ( h_0 - h_1), \quad 
\nu \Delta h_0  = \nu( - h_0 - h_1)\quad \mbox{and} \quad
i \mu ( 1 + \Lambda) h_0  =i \mu  h_1. 
$$
On the other hand, we have  using \eqref{CRhn}
$$
\Tc( h_0 + \beta h_1)h_0 =  \chi_{0000} h_0 
 + |\beta|^2 \chi_{1100} h_0 + \beta \chi_{0110} h_1. 
$$
Hence the equation \eqref{eqmode} is satisfied if we have 
$$
\left|
\begin{array}{l}
 \chi_{0000} + |\beta|^2 \chi_{1100} = \nu - \kappa + \lambda, \\[1ex]
\beta \chi_{0110} h_1 = \nu + \kappa - i \mu, 
\end{array}
\right.
$$
which is a solvable equation in $\beta$ and $\lambda$ as $\chi_{1100}$ is non zero. 
\end{proof}
To make the equation \eqref{eqmode} appear, we proceed as follows: we consider the equation 
$$
i \partial_s f = \Tc[F]f
$$
For $N$, $m$, $\gamma$ and $c$ depending on the time, 
we  make the change of variable 
$$
f =  e^{ i \gamma}\Sc_N     e^{i m |\xi|^2} e^{i c \Delta }  g, 
$$
where $\Sc_N$ is given by \eqref{ScN}. 
\end{proof}
\begin{proposition}
Let $\kappa$, $\nu$, $N$, $c$, $m$, $\lambda$ and $\gamma$ be a given functions of $s$. Let $f$ and $g$, $F$ and $G$  be linked by the relation
$$
f = e^{ i \gamma} \Sc_N     e^{i m |y|^2} e^{i c \Delta}  g, \quad \mbox{and} \quad F = e^{ i \gamma} \Sc_N     e^{i m |y|^2} e^{i c \Delta}  G,
$$
Then 
$$
i \partial_s f = \Tc[F]f \quad \Longleftrightarrow\quad 
i \partial_s g = \nu \Delta g + i \mu ( 1 + \Lambda)g + \kappa |\xi|^2  g + \Tc[G]g - \lambda g
$$
if and only if 
\begin{equation}
\label{newmodul}
\left|
\begin{array}{rcl}
\frac{N_s}{N} &=& 4 c \kappa + \mu, \\ 
m_s &=& (1 + 4 c m) \kappa,\\
c_s &=& \nu -  4 \kappa  c^2, \\
\gamma_s &=& -\lambda.  
\end{array}
\right.
\end{equation}

\end{proposition}
\begin{proof}
To obtain the equation for $g$, we proceed step by step. Let us first assume 
$$
f =   \Sc_N v, \quad \mbox{and} \quad F = \Sc_N V. 
$$
Then we have 
$$
i \partial_s f = i \partial_s    \Sc_N v  =   \Sc_N  i \partial_s v  - i\frac{N_s}{N} \Sc_N ( 1 + \Lambda) v. 
$$
and 
$$
\Tc[F]f = \Sc_N \Tc[V]v. 
$$
We thus find the equation 
$$
i \partial_s v = i \frac{N_s}{N} ( 1 + \Lambda) v + \Tc[V]v. 
$$
Now assume 
$$
v = e^{i m |\xi|^2} w \quad \mbox{and} \quad V = e^{i m |\xi|^2} W
$$
We have
$$
i \partial_s w - m_s |\xi|^2 w  = i \frac{N_s}{N} e^{- i m |\xi|^2}( 1 + \Lambda)e^{i m |\xi|^2} w  + \Tc [W]w. 
$$
Hence using \eqref{C3}, 
$$
i \partial_s w - m_s |\xi|^2 w  = i \frac{N_s}{N} ( 1 + \Lambda)w 
 + \Big( m_s - 2 m  \frac{N_s}{N}\Big)  |\xi|^2 w
 + \Tc [W]w.
$$
Let us set 
$$
\kappa = m_s - 2 m  \frac{N_s}{N}
$$
and $w = e^{i c \Delta + i \gamma} g$, $W =  e^{i c \Delta} G$. 
We find using \eqref{C2} and \eqref{C4}
\begin{eqnarray*}
i \partial_s g - \gamma_s&=& c_s \Delta g +  i \frac{N_s}{N} e^{- i c \Delta }( 1 + \Lambda) e^{i c \Delta} g
 + \kappa e^{- i c \Delta }  |\xi|^2e^{i c \Delta}  g 
 + \Tc[G]g\\
 &=& c_s \Delta g +  i \frac{N_s}{N} ( 1 + \Lambda)  g
 + 2 c \frac{N_s}{N}  \Delta g
 +  \kappa |\xi|^2 g   - 4 \kappa c^2 \Delta g -   4 i \kappa c ( 1 + \Lambda)g 
 + \Tc[G]G\\
 &=& (c_s + 2 c \frac{N_s}{N}  - 4 \kappa c^2) \Delta g + i ( \frac{N_s}{N}  - 4 \kappa c) ( 1 + \Lambda) g + \kappa |\xi|^2 g + \Tc[G]g. 
\end{eqnarray*}
and we obtain the result
\end{proof}
Now by using Proposition \ref{lemnorms} we thus have built solutions to the linear CR equation with $\Hc^1$ norm growing like 
$$
\Norm{\Sc_N     e^{i m |x|^2} e^{i c \Delta}  h_0}{\Hc^1} = N^{2 }\big( 1 + 4 c^2)  + \frac{1}{N^{2}} ((1 + 4cm)^2 + 4m^2\big), 
$$
where $N$, $m$ and $c$ solve \eqref{newmodul}. 
With the simplest example $c = m = 0$, $N = e^{\mu s}$, we obtain the following result: 
\begin{proposition}
Let $\mu > 0$ and $N = e^{\mu s}$. Then there exists $\beta \in i \R$ and $\lambda \in \R$ such that 
 $f(s,x) = e^{- is \lambda}\Sc_N h_0$ and $F(s,x) = e^{- is \lambda} \Sc_N (h_0 + \beta h_1)$ satisfy \eqref{ft} and \eqref{Ft} and provide a solution to the linear CR equation. 
\end{proposition}

Many solutions under the previous form can be constructed, as well as solutions obtained by modulating parameters with the other invariant laws of CR (see table 1). In each case, it provides examples of weakly turbulent solution for linear time dependent equation with pseudo-differential of order  $0$ perturbation and by using Theorem \ref{popop}, examples of smooth potential producing growth of Sobolev norms. The complete classification of all these solutions as well as their genericity is clearly out of the scope of this paper.


\begin{appendix}



\section{Integral and norms of radial Hermite functions\label{hermint}}


\begin{proposition}
For $n,k \geq 0$, we have 

\begin{equation}
\label{positivity}
(h_n, h_0  h_k)_{L^2} = \frac{2}{\sqrt{\pi}}\left( \frac{1}{3}\right)^{n + 1}  \sum_{p + q= k}\left( \frac{1}{3}\right)^{q}  \frac{1}{p! q!} \frac{(n + q)!}{(n-p)!} >0, 
\end{equation}
and
\begin{equation}
\label{positivity2}
(h_n, h_0 h_0 h_k)_{L^2} = \frac{1}{ \pi} { n + k  \choose k} \left(\frac{1}{2}\right)^{n+1 + k} >0. 
\end{equation}
\end{proposition}
\begin{proof}
Using \eqref{generhn}, 
and as $h_0(x) = \frac{1}{\sqrt{\pi}} e^{- \frac{|x|^2}{2}}$, 
$$
\sum_{n = 0}^{\infty} t^k h_0(x) h_k(x) =\frac{1}{\pi}\frac{1}{1 - t} e^{- \frac{ t |x|^2 }{1 - t}}e^{- |x|^2} =  \frac{1}{\pi}\frac{1}{1 - t} e^{- \frac{|x|^2}{1 - t}}. 
$$
We thus have for real numbers $r$ and $s$ such that $|r| < 1$ and $|t |< 1$, 
\begin{eqnarray*}
\sum_{n,k = 0}^\infty r^n t^k ( h_n,h_0 h_k)_{L^2} &=& \frac{1}{\pi^{\frac{3}{2}}}\frac{1}{(1 - t)(1 - r)}  \int_{\R^2}  e^{- \frac{|x|^2}{1 - t}} e^{- \frac{(1+ r)}{2(1 - r)}|x|^2} \dd x\\
&=& \frac{1}{\pi^{\frac{3}{2}}}\frac{1}{(1 - t)(1 - r)}  \int_{\R^2}  e^{- \frac{(3 - r - t - tr)}{2(1 - r)(1 - t)}|x|^2} \dd x\\
&=&  \frac{1}{\pi^{\frac{3}{2}}}\frac{2}{3 - r -t - tr}  \int_{\R^2}  e^{- |x|^2} \dd x = \frac{2}{\sqrt{\pi}} \left( \frac{1}{3 - t - r- tr}\right). 
\end{eqnarray*}
Hence we have 
$$
\frac{\sqrt{\pi}}{2}\sum_{n,k = 0}^\infty r^n t^k ( h_n,h_0 h_k)_{L^2} 
= \frac{1}{(3 - t)} \left( \frac{1}{1 - r\frac{1 + t}{3 - t}}\right) = \frac{1}{3 - t} \sum_{n = 0}^{\infty} r^n \left(\frac{1 + t}{3 - t}\right)^n. 
$$
By letting $t$ be fixed such that $|t | < 1$ and considering $r$ small enough, we deduce that 
$$
F_n(t) := \sum_{k = 0}^\infty t^k ( h_n,h_0 h_k)_{L^2} = \frac{2}{\sqrt{\pi}}\frac{(1+ t)^n}{(3 - t)^{n+1}}.  
$$
This shows that all the coefficients of the developpement are positive, and we have explicitely for $n \geq k$: 
\begin{eqnarray*}
( h_n,h_0 h_k)_{L^2} &=& \frac{2}{\sqrt{\pi}k!}\left.\frac{\dd^k}{\dd t^k} F_n(t) \right|_{t = 0} \\
&=&\frac{2}{\sqrt{\pi}}\left.\frac{1}{k!}\sum_{p + q= k} {k \choose p} \frac{n!}{(n-p)!} (1 + t)^{n - p} \frac{(n +  q)!}{n!}( 3 - t)^{- n - 1 - q} \right|_{t = 0}\\
&=& \frac{2}{\sqrt{\pi}}\left( \frac{1}{3}\right)^{n + 1}  \sum_{p + q= k}\left( \frac{1}{3}\right)^{q}  \frac{1}{p! q!} \frac{(n + q)!}{(n-p)!}  > 0, 
\end{eqnarray*} and we deduce the case $n < k$ by symmetry $( h_n,h_0 h_k)_{L^2} = ( h_k,h_0 h_n)_{L^2}$. This proves \eqref{positivity}. 

To prove \eqref{positivity2} we proceed in a similar way, we have 

$$
\sum_{n = 0}^{\infty} t^k h_0(x)^2 h_k(x) =\frac{1}{\pi^{\frac{3}{2}}}\frac{1}{1 - t} e^{- \frac{ t |x|^2 }{1 - t}}e^{- \frac{3}{2}|x|^2} =  \frac{1}{\pi^\frac32}\frac{1}{1 - t} e^{- \frac{(3 - t)|x|^2}{2(1 - t)}}. 
$$
We thus have for $|r| < 1$ and $|t |< 1$, using \eqref{generhn} 
\begin{eqnarray}
\sum_{n,k = 0}^\infty r^n t^k ( h_n,h_0^2 h_k)_{L^2} &=& \frac{1}{\pi^{2}}\frac{1}{(1 - t)(1 - r)}  \int_{\R^2}  e^{- \frac{(3 - t) |x|^2}{2(1 - t)}} e^{- \frac{(1+ r)}{2(1 - r)}|x|^2} \dd x\nonumber \\
&=& \frac{1}{\pi^{2}}\frac{1}{(1 - t)(1 - r)}  \int_{\R^2}  e^{- \frac{(2 - r - t)}{(1 - r)(1 - t)}|x|^2} \dd x \nonumber \\
&=&  \frac{1}{\pi^{2}}\frac{1}{2 - r -t}  \int_{\R^2}  e^{- |x|^2} \dd x = \frac{1}{\pi} \left( \frac{1}{2 - t - r}\right). \label{deme}
\end{eqnarray}
Hence we have 
$$
\pi\sum_{n,k = 0}^\infty r^n t^k ( h_n,h_0^2 h_k)_{L^2} 
= \frac{1}{(2 - t)} \left( \frac{1}{1 - r\frac{1}{2 - t}}\right) = \frac{1}{2 - t} \sum_{n = 0}^{\infty} r^n \left(\frac{1}{2 - t}\right)^n. 
$$
By letting $t$ fix such that $|t | < 1$ and considering $r$ small enough, we deduce that 
$$
F_n(t) := \sum_{k = 0}^\infty t^k ( h_n,h_0^2 h_k)_{L^2} = \frac{1}{\pi}\frac{1}{(2 - t)^{n+1}}.  
$$
This shows that all the coefficients of the developpement are positive, and we have explicitely: 
$$
( h_n,h_0^2 h_k)_{L^2} =\frac{1}{ \pi k!}\left.\frac{\dd^k}{\dd t^k} F_n(t) \right|_{t = 0} =  \frac{1}{ \pi k!}\frac{(n + k)!}{n!} \left(\frac{1}{2}\right)^{n+1 + k}, 
$$
which shows the result.  
\end{proof}
\section{A backward Gr\"onwall inequality}

\begin{lemma}
\label{gr}
Let $s_0 >0$ and $M >0$. Assume that $\beta(s) > 0$ and $\alpha(s)$ are functions defined on $(s_0,M)$, and that $u(s)$ satisfies
$$
u(s) \leq \alpha(s) + \int_s^M \beta(\sigma) u(\sigma) \dd \sigma. 
$$ 
Then we have 
$$
u(s) \leq  \alpha(s) + \int_{s}^M \alpha(\sigma) \beta(\sigma) \exp \left( \int^{\sigma}_{s}  \beta(\tau) \dd \tau\right) \dd \sigma
$$
\end{lemma}
\begin{proof}
Let $v(s) = u(M-s + s_0)$, $\tilde \alpha(s) = \alpha( M - s + s_0)$ and $\tilde \beta(s) = \alpha( M - s + s_0)$ which are defined on $(s_0,M)$. We have 
$$
v(s) \leq \tilde \alpha(s) + \int_{M - s + s_0}^M  \beta(\sigma) u(\sigma) \dd \sigma
=  \tilde \alpha(s) + \int_{s_0}^s  \tilde \beta(\sigma) v(\sigma) \dd \sigma
$$
By the classical Gr\"onwall inequality, we have 
\begin{eqnarray*}
v(s) &\leq& \tilde \alpha(s) + \int_{s_0}^s \tilde \alpha(\sigma) \tilde \beta(\sigma) \exp \left( \int_{\sigma}^s \tilde \beta(\tau) \dd \tau\right) \dd \sigma \\
&=& \tilde \alpha(s) + \int_{M - s +s_0}^M \alpha(\sigma) \beta(\sigma) \exp \left( \int_{M - \sigma + s_0}^s  \beta(M - \tau + s_0) \dd \tau\right) \dd \sigma
\\
&=& \tilde \alpha(s) + \int_{M - s +s_0}^M \alpha(\sigma) \beta(\sigma) \exp \left( \int^{\sigma}_{M - s + s_0}  \beta(\tau) \dd \tau\right) \dd \sigma
\end{eqnarray*}
from which we deduce the result. 
\end{proof}

\end{appendix}

\end{document}